\documentclass[reqno]{amsart}
\usepackage[american]{babel}
\usepackage{amssymb,euscript}
\numberwithin{equation}{section}
\newtheorem{theorem}{Theorem}[section]
\newtheorem{lemma}[theorem]{Lemma}
\newtheorem{corollary}[theorem]{Corollary}
\theoremstyle{definition}
\newtheorem{definition}[theorem]{Definition}
\theoremstyle{remark}
\newtheorem{remark}[theorem]{Remark}

\DeclareMathOperator{\sgn}{sign}

\begin{document}

\title[Zakharov--Kuznetsov Equation]
{Initial-Boundary Value Problems in a Half-Strip for Two-Dimensional Zakharov--Kuznetsov Equation}

\author[A.V.~Faminskii]{Andrei~V.~Faminskii}

\thanks{The publication was financially supported by the Ministry of Education and Science of the Russian Federation (the Agreement 02.A03.21.0008 and the Project 1.962.2017/PCh)}

\address{Peoples' Friendship University of Russia (RUDN University), 6 Miklukho--Maklaya Street, Moscow, 117198, Russian Federation}

\email{afaminskii@sci.pfu.edu.ru}



\date{}

\begin{abstract}
Initial-boundary value problems in a half-strip with different types of boundary conditions for two-dimensional Zakharov--Kuznetsov equation are considered. Results on global existence, uniqueness and long-time decay of weak and regular solutions are established.
\end{abstract}

\maketitle

\section{Introduction. Description of main results}\label{S1}
The two dimensional Zakharov--Kuznetsov equation (ZK)
\begin{equation}\label{1.1}
u_t+bu_x+u_{xxx}+u_{xyy}+uu_x=f(t,x,y)
\end{equation}
($b$ is a real constant) is a reduction of the three-dimensional one which was derived in \cite{ZK} for description of  ion-acoustic waves in magnetized plasma. Now this equation is considered as a model of two-dimensional nonlinear waves in dispersive media propagating in one preassigned ($x$) direction with deformations in the transverse ($y$) direction. A rigorous derivation of the ZK model can be found, for example, in \cite{H-K, LLS}. It is one of the variants of multi-dimensional generalizations for Korteweg--de~Vries equation (KdV) 
$u_t+bu_x+u_{xxx}+uu_x=f(t,x)$.

The theory of solubility and well-posedness for ZK equation and its generalizations is most developed for the pure initial-value problem. For the considered two-dimensional case the corresponding results in different functional spaces can be found in \cite{S, F89, F95, BL, LP09, LP11, RV, FLP, BJM, GH, MP, FP, G}. For initial-boundary value problems the theory is most developed for domains of a type $I\times \mathbb R$, where $I$ is an interval (bounded or unbounded) on the variable $x$, that is, the variable $y$ varies in the whole line (\cite{F02, F07, FB08, F08, ST, F12, DL}). 

On the other hand, from the physical point of view boundary-value problems for this equation in domains there the variable $y$ varies in a bounded interval seem at least the same important.
Unfortunately certain technique developed for the case $y\in\mathbb R$ (especially related to profound investigation of the corresponding linear equation) up to this moment is extended to the case of bounded $y$ only partially. An initial-boundary value problem in a strip $\mathbb R\times (0,L)$ with periodic boundary conditions was considered in \cite{LPS} for ZK equation and local well-posedness result was established in the spaces $H^s$ for $s>3/2$. This result was improved in \cite{MP} where $s\geq 1$, in addition, in the space $H^1$ appropriate conservation laws provided global well-posedness. 

Another way of the study is based on the use of certain weighted spaces. Initial-boundary value problems in a strip $\mathbb R\times (0,L)$ with homogeneous boundary conditions of different types -- Dirichlet, Neumann or periodic -- were considered in \cite{BF13} for ZK equation with more general nonlinearity and results on global well-posedness in classes of weak solutions with power weights at $+\infty$ were established. Similar results in the case of exponential weights for ZK equation itself under homogeneous Dirichlet boundary conditions can be found in \cite{F15-2}. Global well-posedness results for ZK equation with certain parabolic regularization also for the initial-boundary value problem in a strip $\mathbb R\times (0,L)$ with homogeneous Dirichlet boundary conditions were obtained in \cite{F15-1, F15-2, L15, L16}. Global well-posedness results for a bounded rectangle can be found in \cite{DL, STW}.

An initial-boundary value problem in a half-strip $\mathbb R_+\times (0,L)$ with homogeneous Dirichlet boundary conditions was studied in \cite{LT,L13} and global well-posedness in Sobolev spaces with exponential weights when $x\to +\infty$ was proved.  

In the present paper we consider initial-boundary value problems in a domain $\Pi_T^+=(0,T)\times\Sigma_+$, where $\Sigma_+=\mathbb R_+\times (0,L)=\{(x,y): x>0, 0<y<L\}$ is a half-strip of a given width $L$ and $T>0$ is arbitrary, for equation \eqref{1.1}
with an initial condition
\begin{equation}\label{1.2}
u(0,x,y)=u_0(x,y),\qquad (x,y)\in\Sigma_+,
\end{equation}
boundary condition 
\begin{equation}\label{1.3}
u(t,0,y)=\mu(t,y),\qquad (t,y)\in B_{T}=(0,T)\times (0,L),
\end{equation}
and boundary conditions for $(t,x)\in \Omega_{T,+}=(0,T)\times\mathbb R_+$ of one of the following four types: 
\begin{equation}\label{1.4}
\begin{split}
\mbox{whether}\qquad &a)\mbox{ } u(t,x,0)=u(t,x,L)=0,\\
\mbox{or}\qquad &b)\mbox{ } u_y(t,x,0)=u_y(t,x,L)=0,\\
\mbox{or}\qquad &c)\mbox{ } u(t,x,0)=u_y(t,x,L)=0,\\
\mbox{or}\qquad &d)\mbox{ } u \mbox{ is an  $L$-periodic function with respect to $y$.}
\end{split}
\end{equation}
We use the notation "problem \eqref{1.1}--\eqref{1.4}" for each of these four cases.

The main results consist of theorems on global solubility and well-posedness in classes of weak and regular solutions in certain weighted  at $+\infty$ Sobolev spaces. Both power and exponential weights are allowed. We consider homogeneous boundary conditions when $y=0$, $y=L$ in the cases a)--c) and non-homogeneous one when $x=0$ since didn't succeed to find any specific smoothness properties of solutions on the planes $y=\text{const}$ in comparison with the planes $x=\text{const}$.

Besides that, results on large-time decay of small solutions similar to the ones from \cite{LT, L13} when $\mu\equiv 0$, $f\equiv 0$, are established in the cases a) and c).

All global existence results are based on conservation laws, which in the case $\mu\equiv 0$, $f\equiv 0$ for smooth solutions are written as follows:
\begin{equation}\label{1.5}
\frac{d}{dt}\iint_{\Sigma_+} u^2\,dxdy + \int_0^L u_x^2\big|_{x=0}\,dy=0,
\end{equation}
\begin{equation}\label{1.6}
\frac{d}{dt}\iint_{\Sigma_+} \bigl(u_x^2+u_y^2-\frac13 u^3\bigr)\,dxdy + 
\int_0^L (u_{xx}^2+bu_x^2)\big|_{x=0}\,dy=0.
\end{equation}
Besides that, we use the local smoothing effect which in the most simple form can be written as
\begin{equation}\label{1.7}
\int_0^T\!\!\int_0^r\!\!\int_0^L (u_x^2+u_y^2)\,dxdy \leq c(r,\|u_0\|_{L_2(\Sigma_+)}).
\end{equation}

In what follows (unless stated otherwise) $j$, $k$, $l$, $m$, $n$ mean non-negative integers, $p\in [1,+\infty]$, $s\in\mathbb R$.  For any multi-index $\alpha=(\alpha_1,\alpha_2)$ let $\partial^\alpha =\partial^{\alpha_1}_{x}\partial^{\alpha_2}_{y}$, let
\begin{equation*}
|D^k\varphi|=\Bigl(\sum_{|\alpha|\leq k}(\partial^\alpha \varphi)^2\Bigr)^{1/2}, \qquad
|D\varphi|=|D^1\varphi|.
\end{equation*}
Let $L_{p,+}=L_p(\Sigma_+)$, $W_{p,+}^k=W_p^k(\Sigma_+)$, $H^s_+=H^s(\Sigma_+)$.

Introduce special function spaces taking into account boundary conditions \eqref{1.4}. Let 
$\Sigma= \mathbb R\times (0,L)$, 
$\widetilde{\EuScript S}(\overline{\Sigma})$ be a space of infinitely smooth on $\overline{\Sigma}$ functions $\varphi(x,y)$ such that $\displaystyle{(1+|x|)^n|\partial^\alpha\varphi(x,y)|\leq c(n,\alpha)}$ for any $n$, multi-index $\alpha$, $(x,y)\in \overline{\Sigma}$ and $\partial_y^{2m}\varphi\big|_{y=0} =\partial_y^{2m}\varphi\big|_{y=L}=0$ in the case a), $\partial_y^{2m+1}\varphi\big|_{y=0} =\partial_y^{2m+1}\varphi\big|_{y=L}=0$ in the case b), $\partial_y^{2m}\varphi\big|_{y=0} =\partial_y^{2m+1}\varphi\big|_{y=L}=0$ in the case c), $\partial_y^{m}\varphi\big|_{y=0} =\partial_y^{m}\varphi\big|_{y=L}$ in the case d) for any $m$.

Let $\widetilde H^s$ be the closure of $\widetilde{\EuScript S}(\overline{\Sigma})$ in the norm $H^s(\Sigma)$ and $\widetilde H_+^s$ be the restriction of $\widetilde H^s$ on $\Sigma_+$.

It is easy to see, that $\widetilde H^0_+=L_{2,+}$; for $j \geq 1$ in the case a) $\widetilde H^j_+=\{\varphi\in H^j_+: \partial_y^{2m}\varphi|_{y=0}=\partial_y^{2m}\varphi|_{y=L}=0, \ 2m<j\}$, in the case b) $\widetilde H^j_+=\{\varphi\in H^j_+: \partial_y^{2m+1}\varphi|_{y=0}=\partial_y^{2m+1}\varphi|_{y=L}=0,\ 2m+1<j\}$,  in the case d) $\widetilde H^j_+=\{\varphi\in H^j_+: \partial_y^{m}\varphi|_{y=0}=\partial_y^{m}\varphi|_{y=L}, \ m<j\}$.  

We also use an anisotropic Sobolev space $\widetilde H^{(0,k)}_+$ which is defined as the restriction on $\Sigma_+$ of a space $\widetilde H^{(0,k)}$, where the last space is the closure of $\widetilde{\EuScript S}(\overline{\Sigma})$ in the norm $\sum\limits_{m=0}^k \|\partial_y^m\varphi\|_{L_2(\Sigma)}$.

We say that $\rho(x)$ is an admissible weight function if $\rho$ is an infinitely smooth positive function on $\overline{\mathbb R}_+$ such that $|\rho^{(j)}(x)|\leq c(j)\rho(x)$ for each natural $j$ and all $x\geq 0$. Note that such a function has not more than exponential growth and not more than exponential decrease at $+\infty$. It was shown in \cite{F12} that $\rho^s(x)$ for any $s\in\mathbb R$ is also an admissible weight function. Any exponent $e^{2\alpha x}$ as well as $(1+x)^{2\alpha}$ are admissible weight functions. 

As an another important example of admissible functions, we define 
$\rho_0(x)\equiv 1+ \frac{2}{\pi}\arctan x$. Note that both $\rho_0$ and $\rho'_0$ are admissible weight functions.

For an admissible weight function $\rho(x)$ let $\widetilde H^{k,\rho(x)}_+$ be a space of functions $\varphi(x,y)$ such that $\varphi\rho^{1/2}(x)\in \widetilde H^k_+$. Similar definitions are used for $\widetilde H^{(0,k),\rho(x)}_+$, $H^{k,\rho(x)}_+$. Let $L_{2,+}^{\rho(x)}= \widetilde H^{0,\rho(x)}_+= \{\varphi(x,y): \varphi\rho^{1/2}(x)\in L_{2,+}\}$. Obviously, $L_{2,+}^{\rho_0(x)}=L_{2,+}$.

We construct solutions to the considered problems in spaces $X_w^{k,\rho(x)}(\Pi_T^+)$ and $X^{k,\rho(x)}(\Pi_T^+)$ for admissible weight functions $\rho(x)$, such that $\rho'(x)$ are also admissible weight functions, consisting of functions $u(t,x,y)$ such that in the case  $X_w^{k,\rho(x)}(\Pi_T^+)$ 
\begin{equation}\label{1.8}
\partial_t^j u\in C_w([0,T]; \widetilde H_+^{k-3j,\rho(x)})\cap L_2(0,T;\widetilde H_+^{k-3j+1,\rho'(x)})
\end{equation}
(the symbol $C_w$ denotes the space of  weakly continuous mappings) for $k-3j\geq 0$
(let $X_w^{\rho(x)}(\Pi_T^+)=X_w^{0,\rho(x)}(\Pi_T^+)$), while in the case $X^{k,\rho(x)}(\Pi_T^+)$ the weak continuity with respect to $t$ in \eqref{1.8} is substituted by the strong one. 

Define also
\begin{equation}\label{1.9}
\lambda^+(u;T) =
\sup_{x_0\geq 0}\int_0^T\! \int_{x_0}^{x_0+1}\! \int_0^L u^2\,dydxdt.
\end{equation}

For description of properties of the boundary data $\mu$ introduce anisotropic functional spaces. Let $B=\mathbb R^t \times (0,L)$. Define the functional space $\widetilde{\EuScript S}(\overline{B})$ similarly to $\widetilde{\EuScript S}(\overline{\Sigma})$, where the variable $x$ is substituted by $t$. Let $\widetilde H^{s/3,s}(B)$ be the closure of $\widetilde{\EuScript S}(\overline{B})$ in the norm $H^{s/3,s}(B)$. 

More exactly, let $\psi_l(y)$, $l=1,2\dots$, be the orthonormal in $L_2(0,L)$ system of the eigenfunctions for the operator $(-\psi'')$ on the segment $[0,L]$ with corresponding boundary conditions   $\psi(0)=\psi(L)=0$ in the case a), $\psi'(0)=\psi'(L)=0$ in the case b), $\psi(0)=\psi'(L)=0$ in the case c), $\psi(0)=\psi(L),\psi'(0)=\psi'(L)$ in the case d), $\lambda_l$ be the corresponding eigenvalues. Such systems are well-known and can be written in trigonometric functions.

For any  $\mu\in \widetilde{\EuScript S}(\overline{B})$, $\theta\in\mathbb R$ and $l$ let
\begin{equation}\label{1.10}
\widehat\mu(\theta,l) \equiv \iint_B e^{-i\theta t}\psi_l(y)\mu(t,y)\,dtdy.
\end{equation}
Then the norm in $H^{s/3,s}(B)$ is defined as $\Bigl(\sum\limits_{l=1}^{+\infty} 
\bigl\| (|\theta|^{2/3}+l^2)^{s/2}\widehat\mu(\theta,l)\bigr\|_{L_2(\mathbb R^\theta)}^2\Bigr)^{1/2}$ and the norm in $H^{s/3,s}(I\times (0,L))$ for any interval $I\subset \mathbb R$ as the restriction norm. 

The use of these norm is justified by the following fact. Let $v(t,x,y)$ be the appropriate solution to the initial value problem
$$
v_t+v_{xxx}+v_{xyy}=0,\qquad v\big|_{t=0}=v_0.
$$
Then according to \cite{F08} uniformly with respect to $x\in \mathbb R$
\begin{equation}\label{1.11}
\bigl\|D_t^{1/3}v\bigr\|_{H_{t,y}^{s/3,s}(\mathbb R^2)}^2+
\bigl\|\partial_x v\bigr\|_{H_{t,y}^{s/3,s}(\mathbb R^2)}^2+
\bigl\|\partial_y v\bigr\|_{H_{t,y}^{s/3,s}(\mathbb R^2)}^2
\sim \|v_0\|_{H^s(\mathbb R^2)}^2
\end{equation}
(here $D^\alpha$ denotes the Riesz potential of the order $-\alpha$).

Introduce the notion of weak solutions to the considered problems.

\begin{definition}\label{D1.1}
Let $u_0\in L_{2,+}$, $\mu\in L_2(B_T)$, $f\in L_1(0,T;L_{2,+})$. A function $u\in L_\infty(0,T;L_{2,+})$ is called a generalized solution to problem \eqref{1.1}--\eqref{1.4} if for any function $\phi\in L_2(0,T;\widetilde H_+^2)$, such that $\phi_t, \phi_{xxx}, \phi_{xyy}\in L_2(\Pi_T^+)$, $\phi\big|_{t=T}=0$, $\phi\big|_{x=0} =\phi_x\big|_{x=0}=0$, the following equality holds:
\begin{multline}\label{1.12}
\iiint_{\Pi_T^+}\Bigl[u(\phi_t+b\phi_x+\phi_{xxx}+\phi_{xyy}) +\frac 12 u^2 \phi_x +f\phi\Bigr]\,dxdydt \\
+\iint_{\Sigma_+} u_0\phi\big|_{t=0}\,dxdy +
\iint_{B_T} \mu\phi_{xx}\big|_{x=0}\,dydt =0.
\end{multline}
\end{definition}

\begin{remark}\label{R1.1}
Note that the integrals in \eqref{1.12} are well defined (in particular, since $\phi_x\in L_2(0,T;H_+^2)\subset L_2(0,T;L_{\infty,+})$).
\end{remark}

Now we can formulate the main results of the paper concerning existence and uniqueness.

\begin{theorem}\label{T1.1}
Let $u_0\in L_{2,+}^{\rho(x)}$, $f\in L_1(0,T; L_{2,+}^{\rho(x)})$ for certain $T>0$ and an admissible weight function $\rho(x)$, such that $\rho'(x)$ is also an admissible weight function. Let $\mu\in \widetilde H^{s/3,s}(B_T)$ for certain $s>3/2$. Then there exists a weak solution to problem \eqref{1.1}--\eqref{1.4} $u \in X_w^{\rho(x)}(\Pi_T^+)$, moreover, $\lambda^+(|Du|;T)<+\infty$. If, in addition, $\rho^{1/2}(x)\leq c\rho'(x) \ \forall x\geq 0$, then this solution is unique in $X_w^{\rho(x)}(\Pi_T^+)$.
\end{theorem}

\begin{remark}\label{R1.2}
The exponential weight $\rho(x)\equiv e^{2\alpha x}$, $\alpha>0$, satisfies both existence and uniqueness assumptions. The power weight $\rho(x)\equiv (1+x)^{2\alpha}$, $\alpha>0$, satisfies existence assumptions and for $\alpha\geq 1$ -- uniqueness assumptions. If $u_0\in L_{2,+}$, $f\in L_1(0,T;L_{2,+})$ there exists a weak solution $u\in C_w([0,T];L_{2,+})$, $\lambda^+(|Du|;T)<+\infty$. Note that weak solutions of the type, constructed in Theorem~\ref{T1.1}, are not considered in \cite{LT,L13}.
\end{remark}

\begin{theorem}\label{T1.2}
Let $u_0\in \widetilde H^{1,\rho(x)}_+$, $f\in L_2(0,T; \widetilde H^{1,\rho(x)}_+)$ for certain $T>0$ and an admissible weight function $\rho(x)$, such that $\rho'(x)$ is also an admissible weight function. Let $\mu\in \widetilde H^{2/3,2}(B_T)$, $\mu(0,y)\equiv u_0(0,y)$. Then there exists a weak solution to problem \eqref{1.1}--\eqref{1.4} $u\in X_w^{1,\rho(x)}(\Pi_T^+)$, moreover, $\lambda^+(|D^2 u|;T)<+\infty$. If, in addition, $\rho'(x)\geq 1 \ \forall x\geq 0$, then this solution is unique in $X_w^{1,\rho(x)}(\Pi_T^+)$.
\end{theorem}

\begin{remark}\label{R1.3}
According to \eqref{1.11} the assumptions on the boundary data $\mu$ are natural. The exponential weight $\rho(x)\equiv e^{2\alpha x}$, $\alpha>0$, satisfies both existence and uniqueness assumptions. The power weight $\rho(x)\equiv (1+x)^{2\alpha}$, $\alpha>0$, satisfies existence assumptions and for $\alpha\geq 1/2$ -- uniqueness assumptions. If $u_0\in \widetilde H_+^1$, $f\in L_1(0,T;\widetilde H_+^1)$ there exists a weak solution $u\in C_w([0,T];\widetilde H_+^1)$, $\lambda^+(|D^2u|;T)<+\infty$. Solutions, similar to the ones from Theorem~\ref{T1.2}, are constructed in \cite{L13} in the case of homogeneous Dirichlet boundary conditions and only for exponential weights (which are convenient, but, of course, restrictive). Moreover, for uniqueness results it is also assumed there, that weak solutions are limits of regular ones.
\end{remark}

\begin{theorem}\label{T1.3}
Let $u_0\in \widetilde H^{3,\rho(x)}_+$, $f\in L_1(0,T; \widetilde H_+^{(0,3),\rho(x)})\cap L_2(0,T;\widetilde H^{1,\rho(x)}_+)$, $f_t\in L_1(0,T; L_{2,+}^{\rho(x)})$ for certain $T>0$ and an admissible weight function $\rho(x)$, such that $\rho'(x)$ is also an admissible weight function and $\rho'(x)\geq 1 \ \forall x\geq 0$. Let $\mu\in \widetilde H^{4/3,4}(B_T)$, $\mu(0,y)\equiv u_0(0,y)$. Then there exists a unique solution to problem \eqref{1.1}--\eqref{1.4} $u\in X^{3,\rho(x)}(\Pi_T^+)$.
\end{theorem}

\begin{remark}\label{R1.4}
According to \eqref{1.11} the assumptions on the boundary data $\mu$ are natural. Both the exponential weight $\rho(x)\equiv e^{2\alpha x}$, $\alpha>0$ and the power weight $\rho(x)\equiv (1+x)^{2\alpha}$, $\alpha\geq 1/2$,  satisfy the hypothesis of the theorem. In \cite{LT} for construction of regular solutions only exponential weights are used and only homogeneous Dirichlet boundary conditions are considered. Moreover, what seems the most important, for the constructed regular solutions existence of $u_{yyy}$, lying in weighted $L_2$-spaces uniformly with respect to $t$, is not obtained there in comparison with Theorem~\ref{T1.3}. 
\end{remark}

Next, pass to the decay results. Here we always assume that $f\equiv 0$, $\mu\equiv 0$ and consider boundary conditions \eqref{1.4} only in the cases a) and c). Similarly to \cite{LT,L13} we use for these results only exponential weights.

\begin{theorem}\label{T1.4}
Let $L_0=+\infty$ if $b\leq 0$, and if $b>0$ there exists $L_0>0$, such that in both cases for any $L\in (0,L_0)$ there exist $\alpha_0>0$, $\epsilon_0>0$ and $\beta>0$, such that if $u_0\in L_{2,+}^{e^{2\alpha x}}$ for $\alpha\in (0,\alpha_0]$, $\|u_0\|_{L_{2,+}}\leq\epsilon_0$,  $f\equiv 0$, $\mu\equiv 0$,  in the cases a) and c) in \eqref{1.4} the corresponding unique weak solution $u(t,x,y$) to problem \eqref{1.1}--\eqref{1.4} from the space $X_w^{e^{2\alpha x}}(\Pi_T^+)$ $\forall T>0$ satisfies an inequality
\begin{equation}\label{1.13}
\|e^{\alpha x}u(t,\cdot,\cdot)\|^2_{L_{2,+}}\leq e^{-\alpha\beta t}\|e^{\alpha x}u_0\|_{L_{2,+}}^2\qquad \forall t\geq 0.
\end{equation}
If, in addition, $u_0\in \widetilde H_+^{1,e^{2\alpha x}}$, $u_0(0,y)\equiv 0$, then for certain constant $c$, depending on $\|u_0\|_{\widetilde H_+^{1,e^{2\alpha x}}}$,
\begin{equation}\label{1.14}
\|e^{\alpha x}u(t,\cdot,\cdot)\|_{H^1_+}^2
\leq ce^{-\alpha\beta t}\qquad \forall t\geq 0.
\end{equation}
\end{theorem}

Further, let $\eta(x)$ denotes a cut-off function, namely, $\eta$ is an infinitely smooth non-decreasing function on $\mathbb R$  such that $\eta(x)=0$ when $x\leq 0$, $\eta(x)=1$ when $x\geq 1$, $\eta(x)+\eta(1-x)\equiv 1$.

Let $\widetilde{\EuScript S}(\overline{B}_+)$ be the restriction of $\widetilde{\EuScript S}(\overline{B})$ on $\overline{B}_+=\overline{\mathbb R}_+^t\times [0,L]$.

We drop limits of integration in integrals over the whole half-strip $\Sigma_+$.

\medskip

The following interpolating inequality generalizing the one from \cite{LSU} for weighted Sobolev spaces is crucial for the study.

\begin{lemma}\label{L1.1}
Let $\rho_1(x)$, $\rho_2(x)$ be two admissible weight functions, such that $\rho_1(x)\leq c_0\rho_2(x)$ $\forall x\geq 0$ for some constant $c_0>0$. Then for any $q\in [2,+\infty)$ there exists a constant $c>0$ such that for every function $\varphi(x,y)$, satisfying $|D\varphi|\rho_1^{1/2}(x)\in L_{2,+}$, $\varphi\rho_2^{1/2}(x)\in L_{2,+}$, the following inequality holds:
\begin{equation}\label{1.15}
\bigl\| \varphi\rho_1^s(x)\rho_2^{1/2-s}(x)\bigr\|_{L_{q,+}} \leq c
\bigl\| |D\varphi|\rho_1^{1/2}(x)\bigr\|^{2s}_{L_{2,+}}
\bigl\| \varphi\rho_2^{1/2}(x)\bigr\|^{1-2s}_{L_{2,+}}
+c\bigl\|\varphi\rho_2^{1/2}(x)\bigr\|_{L_{2,+}},
\end{equation}
where $\displaystyle{s=\frac{1}{2}-\frac{1}{q}}$. If $\varphi\big|_{y=0}=0$ or $\varphi\big|_{y=L}=0$, then the constant $c$ in \eqref{1.15} is uniform with respect to $L$.
\end{lemma}

\begin{proof}
For the whole strip $\Sigma=\mathbb R\times (0,L)$ this inequality was proved in \cite{F15-2}. For $\Sigma_+$ the proof is the same.
\end{proof}

\begin{lemma}\label{L1.2}
For an admissible weight function $\rho(x)$ introduce a functional space $H_+^{(-1,0),\rho(x)}= \{\varphi = \varphi_0 +\varphi_{1x}: \varphi_0, \varphi_1 \in L_{2,+}^{\rho(x)}\}$ endowed with the natural norm. Then for $j=1$ and $j=2$
\begin{equation}\label{1.16}
\|\partial_x^j \varphi\|_{L_{2,+}^{\rho(x)}} \leq c(\rho)\bigl( \|\varphi_{xxx}\|_{H_+^{(-1,0),\rho(x)}} +
\|\varphi\|_{H^{j-1,\rho(x)}_+}\bigr).
\end{equation}
\end{lemma}

\begin{proof}
The proof is obvious.
\end{proof}

We also use the following obvious  interpolating inequalities:
\begin{equation}\label{1.17}
\int_0^L \varphi^2\big|_{x=0}\,dy \leq c \Bigl(\iint \varphi_x^2\rho'\,dxdy\Bigr)^{1/2}
\Bigl(\iint \varphi^2\rho\,dxdy\Bigr)^{1/2}+
c\iint \varphi^2\rho\,dxdy,
\end{equation}
where the constant $c$ depends on the properties of an admissible weight function $\rho$,
and
\begin{equation}\label{1.18}
\|\varphi\|_{L_{\infty,+}} \leq c \|\varphi\|_{H^2_+}.
\end{equation}

For the decay results, we need Steklov's inequalities in the following form: for $\psi\in H_0^1(0,L)$,
\begin{equation}\label{1.19}
\int_0^L \psi^2(y)\,dy \leq \frac{L^2}{\pi^2} \int_0^L \bigl(\psi'(y)\bigr)^2\,dy,
\end{equation}
for $\psi\in H^1(0,L)$, $\psi\big|_{y=0}=0$,
\begin{equation}\label{1.20}
\int_0^L \psi^2(y)\,dy \leq \frac{4L^2}{\pi^2} \int_0^L \bigl(\psi'(y)\bigr)^2\,dy.
\end{equation}

The paper is organized as follows. Auxiliary linear problems are considered in Section~\ref{S2}.  Section~\ref{S3} is devoted to the existence results for the original problems. Results on uniqueness and continuous dependence are proved in Section~\ref{S4}. Decay of solutions is studied in Section~\ref{S5}.

\section{Auxiliary linear problems}\label{S2}

Consider an initial-boundary value in $\Pi_T^+$ for a linear equation
\begin{equation}\label{2.1}
u_t+bu_x+u_{xxx}+u_{xyy}=f(t,x,y)
\end{equation}
with initial and boundary conditions \eqref{1.2}--\eqref{1.4}. Weak solutions to this problem are understood similarly to Definition~\ref{D1.1}, moreover, due to the absence of nonlinearity one can take solutions from more wide space $L_2(\Pi_T^+)$.

\begin{lemma}\label{L2.1}
A generalized solution to problem \eqref{2.1}, \eqref{1.2}--\eqref{1.4} is unique in the space $L_2(\Pi_T^+)$.
\end{lemma}

\begin{proof}
The proof is implemented by the standard H\"olmgren's argument. Consider the adjoint problem in $\Pi_T^+$ for an equation
\begin{equation}\label{2.2}
u_t-bu_x-u_{xxx}-u_{xyy}=f(t,x,y)\in C_0^\infty(\Pi_T^+)
\end{equation}
with zero initial data \eqref{1.2}, boundary data \eqref{1.4} and boundary data on $B_T$
\begin{equation}\label{2.3}
u\big|_{x=0}= u_x\big|_{x=0}=0.
\end{equation}
Let $\{\varphi_j(x): j=1,2,\dots\}$ be a set of linearly independent functions complete in the space $\{\varphi \in H^3_+: \varphi(0)=0\}$. We use the Galerkin method and seek an approximate solution in the form $u_k(t,x,y)= \sum\limits_{j,l=1}^k c_{kjl}(t) \varphi_j(x)\psi_l(y)$ (remind that $\psi_l$ are the orthonormal in $L_2(0,L)$ eigenfunctions for the operator $(-\psi'')$ on the segment $[0,L]$ with corresponding boundary conditions) via conditions for $i,m=1,\dots,k$, $t\in [0,T]$
\begin{equation}\label{2.4}
\iint \bigl(u_{kt}\varphi_i(x)\psi_m(y) +u_k(b\varphi_i'\psi_m+\varphi'''_i\psi_m +\varphi'_i\psi''_m)\bigr)\,dxdy -
\iint f\varphi_i\psi_m\,dxdy=0,
\end{equation}
$c_{kjl}(0)=0$. Multiplying \eqref{2.4} by $2c_{kim}(t)$ and summing with respect to $i,m$, we find that
\begin{equation}\label{2.5}
\frac d{dt} \iint u_k^2\,dxdy +\int_0^L u_{kx}^2\big|_{x=0}\,dy = 2 \iint fu_k\,dxdy
\end{equation}
and, therefore,
\begin{equation}\label{2.6}
\|u_k\|_{L_\infty(0,T;L_{2,+})} \leq \|f\|_{L_1(0,T;L_{2,+})}.
\end{equation}
Next, putting in \eqref{2.4} $t=0$, multiplying by $c'_{kim}(0)$ and summing with respect to $i,m$, we derive that $u_{kt}(0)=0$. Then differentiating \eqref{2.4} with respect to $t$, multiplying by $2c_{kim}(t)$ and summing with respect to $i,m$, we find  similarly to \eqref{2.5}, \eqref{2.6} that
\begin{equation}\label{2.7}
\|u_{kt}\|_{L_\infty(0,T;L_{2,+})} \leq \|f_t\|_{L_1(0,T;L_{2,+})}.
\end{equation}
Finally, since $\psi_m^{(2n)}(y)=(-\lambda_m)^n\psi_m(y)$ it follows from \eqref{2.4} that
\begin{multline}\label{2.8}
\iint \bigl(\partial_y^{n} u_{kt}\varphi_i\psi_m^{(n)} +\partial_y^{n} u_k(b\varphi'_m\psi_m^{(n)}+ \varphi'''_i\psi_m^{(n)}+\varphi'_i\psi_m^{(n+2)})\bigr)\,dxdy \\- 
\iint \partial_y^{n}f\varphi_i\psi_m^{(n)}\,dxdy=0,
\end{multline}
which similarly to \eqref{2.5}, \eqref{2.6} yields that for any $n$
\begin{equation}\label{2.9}
\|\partial_y^n u_k\|_{L_\infty(0,T;L_{2,+})} \leq \|\partial_y^n f\|_{L_1(0,T;L_{2,+})}.
\end{equation}
Estimates \eqref{2.6}, \eqref{2.7}, \eqref{2.9} provide existence of a weak solution $u(t,x,y)$ to the considered problem such that $u, u_t, \partial_y^n u \in L_\infty(0,T;L_{2,+})\ \forall n$ in the following sense: for any function $\phi\in L_2(0,T;\widetilde H_+^2)$, such that $\phi_t, \phi_{xxx}, \phi_{xyy}\in L_2(\Pi_T^+)$, $\phi\big|_{t=T}=0$, $\phi\big|_{x=0}=0$, the following equality holds:
\begin{equation}\label{2.10}
\iiint_{\Pi_T^+}\Bigl[u(\phi_t-b\phi_x-\phi_{xxx}-\phi_{xyy}) +f\phi\Bigr]\,dxdydt =0.
\end{equation}
Note, that the traces of the function $u$ satisfy conditions \eqref{1.2} for $u_0\equiv 0$ and \eqref{1.4}. Moreover, it follows from \eqref{2.10} that $\partial_y^n u_{xxx} \in L_\infty(0,T;H^{(-1,0),1}_+)\ \forall n$, therefore, inequality \eqref{1.16} for $j=1$ yields that $\partial_y^n u_x \in L_\infty(0,T;L_{2,+})\ \forall n$ and one more application of \eqref{2.10} yields that $\partial_y^n u_{xxx} \in L_\infty(0,T;L_{2,+})\ \forall n$, the function $u$ satisfies equation \eqref{2.2} a.e. in $\Pi_T^+$ and its traces satisfy \eqref{2.3}. 

The end of the proof of the lemma is standard.
\end{proof}

With the use of Galerkin method we prove one result on solubility of the considered problem in an infinitely smooth case.

\begin{lemma}\label{L2.2}             
Let $u_0\equiv 0$, $\mu\in C_0^\infty(B_+)$, $f\equiv 0$. Then there exists a solution $u(t,x,y)$ to problem \eqref{2.1}, \eqref{1.2}--\eqref{1.4}, such that $\partial^j_t\partial ^\alpha u \in C_b(\overline{\mathbb R}_+^t; L_{2,+})$ for any $j$ and multi-index $\alpha$ (here and further index 'b" means a bounded map).
\end{lemma}

\begin{proof}
Let $v(t,x,y)\equiv u(t,x,y) -\mu(t,y)\eta(1-x)$, then the original problem is equivalent for the problem of \eqref{2.1}, \eqref{1.2}--\eqref{1.4} for the function $v$ with homogeneous initial-boundary conditions and $f\equiv -\mu_t\eta(1-x)-b\mu\eta'(1-x) -\mu\eta'''(1-x) -\mu_{yy}\eta'(1-x)$. 

Seek an approximate solution in the form $v_k(t,x,y)= \sum\limits_{j,l=1}^k c_{kjl}(t) \varphi_j(x)\psi_l(y)$ (the functions $\varphi_j$ are the same as in the proof of Lemma~\ref{L2.1}) via conditions for $i,m=1,\dots,k$, $t\in [0,T]$
\begin{equation}\label{2.11}
\iint (v_{kt}+bv_{kx}+v_{kxxx}+v_{kxyy})\varphi_i(x)\psi_m(y)\,dxdy  -
\iint f\varphi_i\psi_m\,dxdy=0, \quad c_{kjl}(0)=0.
\end{equation}
Multiplying \eqref{2.11} by $2c_{kim}(t)$ and summing with respect to $i,m$, we find that
\begin{equation}\label{2.12}
\frac d{dt} \iint v_k^2\,dxdy +\int_0^L v_{kx}^2\big|_{x=0}\,dy = 
2\iint f v_k\,dxdy.
\end{equation}
Note that $\|v_k(t,\cdot,\cdot)\|_{L_{2,+}}$ doesn't increase if $t\geq T$ for certain $T$. Then the consequent argument from the proof of Lemma~\ref{L2.1} can be applied here (\eqref{2.10} must be substituted by the corresponding analogue of \eqref{1.12}). Thus, first existence of a solution $v$ such that $\partial^j_t\partial^n_y v \in C_b(\overline{\mathbb R}_+^t; L_{2,+})$ for all $j$ and $n$ is obtained; then with the use of induction with respect to $m$ one can find that $\partial^j_t\partial^n_y \partial_x^{3m} v \in C_b(\overline{\mathbb R}_+^t; L_{2,+})$.
\end{proof}

Before the continuation of the study of the problems in the half-strip consider the corresponding problems in the whole strip.

For $u_0\in \widetilde{\EuScript S}(\overline{\Sigma})$ define similarly to \eqref{1.10} for $\xi\in\mathbb R$ and $l$ 
\begin{equation}\label{2.13}
\widehat u_0(\xi,l) \equiv \iint_{\Sigma} e^{-i\xi x}\psi_l(y) u_0(x,y)\,dxdy,
\end{equation}
\begin{equation}\label{2.14}
S(t,x,y;u_0) \equiv \sum_{l=1}^{+\infty} \frac{1}{2\pi}
\int_{\mathbb R}\, e^{it(\xi^3-b\xi+\lambda_l\xi)} e^{i\xi x}\widehat{u}_0(\xi,l)\,d\xi \psi_l(y).
\end{equation}
It is easy to see that for all $s\in\mathbb R$ the function $S(t,x,y;u_0)\in C_b(\mathbb R^t;\widetilde H^s)$ and for any $t\in \mathbb R$
\begin{equation}\label{2.15}
\|S(t,\cdot,\cdot;u_0)\|_{\widetilde H^s} =\|u_0\|_{\widetilde H^s}.
\end{equation}
This property gives an opportunity to extend the notion of the function $S(t,x,y;u_0)$ to any function $u_0\in \widetilde H^s$ for any $s\in\mathbb R$ via closure in the space $C_b(\mathbb R^t;\widetilde H^s)$, then, of course, equality \eqref{2.15} holds.

Let $\varphi_l(\xi) \equiv \xi^3-b\xi+\lambda_l\xi$. This function increases monotonically if $\lambda_l\geq b$  on the whole real line and for $\xi<-\sqrt{(b-\lambda_l)/3}$ and $\xi>\sqrt{(b-\lambda_l)/3}$ if $\lambda_l<b$. Let 
$\kappa_l(\theta)\equiv \varphi^{-1}_l(\theta)$, which is defined for all $\theta$ if $\lambda_l\geq b$ and for $|\theta| \geq 2((b-\lambda_l)/3)^{3/2}$ if $\lambda_l<b$ (then 
$|\kappa_l(\theta)|\geq 2\sqrt{(b-\lambda_l)/3}$).

\begin{lemma}\label{L2.3}
If $u_0\in \widetilde H^s$ for certain $s\in\mathbb R$, then $S(t,x,y;u_0)\in C_b(\mathbb R^x;  \widetilde H^{(s+1)/3,s+1}((-T,T)\times (0,L))$ and for any $x\in\mathbb R$
\begin{equation}\label{2.16}
\|S(\cdot,x,\cdot;u_0)\|_{\widetilde H^{(s+1)/3,s+1}((-T,T)\times (0,L))} \leq c(T) 
\|u_0\|_{\widetilde H^s}.
\end{equation}
\end{lemma}

\begin{proof}
Without loss of generality assume that $u_0\in \widetilde{\EuScript S}(\overline{\Sigma})$.
There exists $l_0$ such that for $l>l_0$ and all $\xi$ and there exists $\xi_0\geq 1$ such that for $|\xi|\geq \xi_0$ and all $l$
\begin{equation}\label{2.17}
\varphi'_l(\xi) =3\xi^2-b+\lambda_l \geq c(\xi^2+l^2).
\end{equation}
Divide $u_0$ into two parts:
\begin{equation}\label{2.18}
u_{00}(x,y)\equiv \sum\limits_{l=1}^{l_0} \EuScript F^{-1}_x\bigr[ \widehat u_0(\xi,l)\eta(\xi_0+1-|\xi|)\bigr](x)\psi_l(y),\quad
u_{01}(x,y)\equiv u_0(x,y)-u_{00}(x,y).
\end{equation}
After the change of variables in the corresponding analog of the integral in \eqref{2.14} $\theta=\varphi_l(\xi)$ (without loss of generality one can assume also that $\varphi^{-1}_l(\theta)=\kappa_l(\theta)$) we derive that for the obviously defined function $\chi(\xi,l)$ (in particular, $\chi(\xi,l)=0$ for $l\leq l_0$, $|\xi|\leq \xi_0$)
\begin{equation}\label{2.19}
S(t,x,y;u_{01})= \sum_{l=1}^{+\infty} \EuScript F^{-1}_t\Bigr[ e^{i\kappa_l(\theta) x}
\widehat u_0(\kappa_l(\theta),l)\kappa'_l(\theta)\chi(\kappa_l(\theta),l)\Bigr](t) \psi_l(y) 
\end{equation}
and uniformly with respect to $x$ 
\begin{multline}\label{2.20}
\|S(\cdot,x,\cdot;u_{01})\|_{H^{(s+1)/3,s+1}(B)}  \\ = 
\Bigl(\sum\limits_{l=1}^{+\infty} 
\bigl\| (|\theta|^{2/3}+l^2)^{(s+1)/2}\widehat u_0(\kappa_l(\theta),l) \kappa'_l(\theta)\chi(\kappa_l(\theta),l)
\bigr\|_{L_2(\mathbb R^\theta)}^2\Bigr)^{1/2} \\ 
=\Bigl(\sum\limits_{l=1}^{+\infty} 
\bigl\| (|\varphi_l(\xi)|^{2/3}+l^2)^{(s+1)/2}\widehat u_0(\xi,l)  (\varphi'_l(\xi))^{-1/2}\chi(\xi,l)
\bigr\|_{L_2(\mathbb R^\xi)}^2\Bigr)^{1/2}  \leq c
\|u_0\|_{\widetilde H^s}^2.
\end{multline}
Finally note that
$$
S(t,x,y;u_{00})= \sum\limits_{l=1}^{l_0} \EuScript F^{-1}_x\bigr[ e^{it\varphi_l(\xi)} \widehat u_0(\xi,l)\eta(\xi_0+1-|\xi|)\bigr](x)\psi_l(y)
$$
 and one can easily show that for any $j$ and $n$ uniformly with respect to $t\in\mathbb R$
$$
\|\partial_t^j\partial_y^n S(t,\cdot,\cdot;u_{00})\|_{L_2} \leq 
c(s,j,n)\|u_0\|_{\widetilde H^s}.
$$
\end{proof}

Next, introduce the notation
\begin{equation}\label{2.21}
K(t,x,y;f) \equiv \int_0^t S(t-\tau,x,y;f(\tau,\cdot,\cdot))\,d\tau.
\end{equation}
Obviously, if $f\in L_1(0,T;\widetilde H^s)$ for certain $s\in\mathbb R$ then $K(t,x,y;f) \in C([0,T];\widetilde H^s)$ and
\begin{equation}\label{2.22} 
\|K(\cdot,\cdot,\cdot;f)\|_{C([0,T];\widetilde H^s)} \leq \|f\|_{L_1(0,T;\widetilde H^s)}.
\end{equation}

\begin{lemma}\label{L2.4}
If $s\in [-1,2]$, $f\in L_2(0,T;\widetilde H^s)$, then the function $K(t,x,y;f)\in C_b(\mathbb R^x;  \widetilde H^{(s+1)/3,s+1}(B_T))$ and for any $x\in\mathbb R$, $t_0\in (0,T]$
\begin{equation}\label{2.23}
\|K(\cdot,x,\cdot;f)\|_{\widetilde H^{(s+1)/3,s+1}(B_{t_0})} \leq c(T) t_0^{1/3-s/6}
\|f\|_{L_2(0,t_0;\widetilde H^s)}.
\end{equation}
\end{lemma}

\begin{proof}
For $s=-1$ it follows from \eqref{2.16} that
$$
\|K(\cdot,x,\cdot;f)\|_{L_2(B_{t_0})} \leq c(T) \|f\|_{L_1(0,t_0;\widetilde H^{-1})} \leq 
c(T)t_0^{1/2} \|f\|_{L_2(0,t_0;\widetilde H^{-1})}.
$$
Next,
\begin{gather*}
\partial_y^3 K(t,x,y;f) = K(t,x,y;\partial_y^3 f),\\
\partial_t K(t,x,y;f) = f(t,x,y) + K(t,x,y;(b\partial_x+\partial_x^3 +\partial_x\partial_y^2)f),
\end{gather*}
and again applying \eqref{2.16} for $s=-1$ we derive that
$$
\|K(\cdot,x,\cdot;f)\|_{\widetilde H^{1,3}(B_{t_0})} \leq c(T) \|f\|_{L_2(0,t_0;\widetilde H^2)}.
$$
For intermediate values of $s$ the result follows by interpolation.
\end{proof}

\begin{remark}\label{R2.1}
If $f\in L_1(0,T;L_2(\Sigma))$, then Lemma~\ref{L2.4} immediately provides that $K_x(t,x,y;f)\in C_b(\mathbb R^x;  L_2(B_T))$ and uniformly with respect to $x\in\mathbb R$
\begin{equation}\label{2.24}
\|K_x(\cdot,x,\cdot;f)\|_{L_2(B_T)} \leq c(T) \|f\|_{L_1(0,T;L_2(\Sigma))}.
\end{equation}
\end{remark}

If $u_0\in L_2(\Sigma)$, $f\in L_1(0,T;L_2(\Sigma))$, then a function
\begin{equation}\label{2.25}
u(t,x,y) \equiv S(t,x,y;u_0) +K(t,x,y;f)
\end{equation}
is a week solution to an initial-boundary value problem in a strip $\Sigma$ to problem \eqref{2.1}, \eqref{1.2} (for $(x,y)\in \Sigma$), \eqref{1.4} (for $(t,x)\in (0,T)\times \mathbb R)$) (see, for example, \cite{BF13} ).

\medskip

In what follows, we need some properties of solutions to an algebraic equation
\begin{equation}\label{2.26}
z^3-(\lambda_l-b) z+p=0, \qquad p=\varepsilon+i\theta \in \mathbb C.
\end{equation}
For $\varepsilon>0$ we denote by $z_0(p,l)$ the unique root of this equation, such
that $\Re z_0<0$.

\begin{lemma}\label{L2.5}
There exists
\begin{equation}\label{2.27}
\lim\limits_{\varepsilon\to +0} z_0(\varepsilon+i\theta,l) =r_0(\theta,l)
=p(\theta,l)+iq(\theta,l),
\end{equation}
where $r_0(\cdot,l)\in C(\mathbb R)$, $r_0(-\theta,l)=\overline{r_0(\theta,l)}$, $p(\theta,l),q(\theta,l)\in \mathbb R$ and
\begin{equation}\label{2.28}
|r_0(\theta,l)| \leq c(|\theta|^{1/3}+\lambda_l^{1/2}+|b|^{1/2}), \qquad c=\text{const}>0.
\end{equation}
If $\lambda_l\geq b$, then
\begin{equation}\label{2.29}
p(\theta,l) \leq -c_0(|\theta|^{1/3}+(\lambda_l-b)^{1/2}). \qquad c_0=\text{const}>0.
\end{equation}
If $\lambda_l<b$, then for $|\theta|\geq 2((b-\lambda_l)/3)^{3/2}$
\begin{equation}\label{2.30}
p(\theta,l) \leq -c_0 \bigl(|\kappa_l(\theta)|-2\sqrt{(b-\lambda_l)/3}\bigr), \qquad c_0=\text{const}>0,
\end{equation}
while for $|\theta|< 2((b-\lambda_l)/3)^{3/2}$
\begin{equation}\label{2.31}
p(\theta,l)=0, \quad |q(\theta,l)|\leq \sqrt{(b-\lambda_l)/3}, \quad
q(\theta,l)=\varphi_l^{-1}(\theta).
\end{equation}
\end{lemma}

\begin{proof}
This lemma evidently follows from the Cardano formula. In particular, if $\lambda_l<b$ then for $|\theta|> 2((b-\lambda_l)/3)^{3/2}$
$$
p(\theta,l) = -\frac{\sqrt{3}}2 \left[
\sqrt[3]{\frac{\theta}2 + \sqrt{\frac{\theta^2}4-\frac{(b-\lambda_l)^3}{27}}} -
\sqrt[3]{\frac{\theta}2 -  \sqrt{\frac{\theta^2}4-\frac{(b-\lambda_l)^3}{27}}}\right],
$$
therefore, it is easily verified that 
$$
p'(\theta,l)\sgn\theta \leq -\frac1{2^{4/3}\sqrt{3}}|\theta|^{-2/3}.
$$ 
Since obviously $\varphi'_l(\xi) \geq \varphi_l^{2/3}(\xi)$ for $|\xi|\geq 2\sqrt{(b-\lambda_l)/3}$ and so $\theta^{-2/3}\geq \kappa'_l(\theta)$ inequality \eqref{2.30} follows.
\end{proof}

Now introduce a special solution of equation \eqref{2.1} for $f\equiv 0$ of "boundary potential" type.

\begin{definition}\label{D2.1}
Let $\mu \in\widetilde{\EuScript S}(\overline{B})$. Define for $x\geq 0$
\begin{equation}\label{2.32}
J(t,x,y;\mu) \equiv \sum\limits_{l=1}^{+\infty} \EuScript F^{-1}_t \Bigl[ e^{r_0(\theta,l)x}\widehat\mu(\theta,l)\Bigr](t) \psi_l(y),
\end{equation}
where $\widehat\mu(\theta,l)$ is given by formula \eqref{1.10}.
\end{definition}

\begin{remark}\label{R2.2}
Since $\widehat J(\theta,x,l;\mu) =   e^{r_0(\theta,l)x}\widehat\mu(\theta,l)$ and $\Re r_0(\theta,l)\leq 0$, then 
$ J(t,x,l;\mu) \in C_b(\overline{\mathbb R}_+^x;\widetilde H^{s/3.s}(B))$ for any $s\in\mathbb R$ and 
\begin{equation}\label{2.33}
\|J(\cdot,\cdot,\cdot;\mu)\|_{ C_b(\overline{\mathbb R}_+^x;\widetilde H^{s/3.s}(B))} \leq 
\|\mu\|_{\widetilde H^{s/3.s}}.
\end{equation}
Therefore, the notion of the function $J(t,x,y;\mu)$ can be extended  in the space $C_b(\overline{\mathbb R}_+^x;\widetilde H^{s/3.s}(B))$ for any function $\mu\in \widetilde H^{s/3.s}(B))$ for certain $s\in\mathbb R$ with conservation of inequality \eqref{2.33}. It is obvious, that $J(t,0,y;\mu)\equiv \mu(t,y)$.

Moreover, in the most important for us case $s\geq 0$ the values $\widehat\mu(\theta,l)$ can be defined directly as limits in $L_2(B)$, for example, of integrals $\displaystyle \int_{-T}^T\!\int_0^L e^{-i\theta t}\psi_l(y)\mu(t,y)\,dtdy$, $T\to +\infty$. Then the function $J(t,x,y;\mu)$ can be equivalently defined simply by formula \eqref{2.32}.
\end{remark}

\begin{lemma}\label{L2.6}
If $\mu\in \widetilde H^{s/3.s}(B)$ for certain $s\geq 0$, then for any $n\leq s$
there exists $\partial_x^n J(t,x,y;\mu)\in  C_b(\overline{\mathbb R}_+^x;\widetilde H^{(s-n)/3.s-n}(B))$ and uniformly with respect to $x\geq 0$
\begin{equation}\label{2.34}
\|\partial_x^n J(\cdot,x,\cdot;\mu)\|_{\widetilde H^{(s-n)/3.s-n}(B)} \leq c(s) 
\|\mu\|_{\widetilde H^{s/3.s}(B)}.
\end{equation}
\end{lemma}

\begin{proof}
The proof is similar to the proof of inequality \eqref{2.33} also with the use of \eqref{2.28}. 
\end{proof}

\begin{lemma}\label{L2.7}
If $\mu\in \widetilde H^{(s+1)/3.s+1}(B)$ for certain $s\geq 0$, then for any $j\leq s/3$ there exists
$\partial_t^j J(t,x,y;\mu)\in C_b(\mathbb R^t;\widetilde H^{s-3j}_+)$ and uniformly with respect to
$t\in\mathbb R$
\begin{equation}\label{2.35}
\|\partial_t^j J(t,\cdot,\cdot;\mu)\|_{\widetilde H_+^{s-3j}} \leq c(s) \|\mu\|_{\widetilde H^{(s+1)/3.s+1}(B)}.
\end{equation}
\end{lemma}

\begin{proof}
Without loss of generality one can assume that $\mu \in\widetilde{\EuScript S}(\overline{B})$. Let $s$ be integer. Then for $3j+n+m=s$
\begin{equation}\label{2.36}
\partial_t^j \partial_x^n \partial_y^m J(t,x,y;\mu) = \sum\limits_{l=1}^{+\infty} \frac1{2\pi}
\int_{\mathbb R} (i\theta)^j r_0^n(\theta,l)e^{it\theta}e^{r_0(\theta,l)x}\widehat\mu(\theta,l)\,d\theta \psi_l^{(m)}(y).
\end{equation}
Divide the expression in the right side of \eqref{2.36} into two parts. Let $l_0$ be such that $\lambda_l<b$ for $l\leq l_0$ and let
$$
I_1 \equiv \sum\limits_{l=1}^{l_0} \frac1{2\pi}
\int_{|\theta|<2((b-\lambda_l)/3)^{3/2} } (i\theta)^j r_0^n(\theta,l)e^{it\theta}e^{r_0(\theta,l)x}\widehat\mu(\theta,l)\,d\theta \psi_l^{(m)}(y)
$$
(it is absent if $\lambda_l\geq b$ $\forall l$) and let $I_2$ be the rest part.

First consider $I_1$. According to \eqref{2.31} $r_0(\theta,l)=iq(\theta,l)$ and changing variables $\xi=q(\theta,l)$ we derive that 
$$
I_1= \sum\limits_{l=1}^{l_0} \frac1{2\pi}
\int_{|\xi|<\sqrt{(b-\lambda_l)/3} }e^{it\varphi_l(\xi)} e^{i\xi x}(i\varphi_l(\xi))^j (i\xi)^n\widehat\mu(\varphi_l(\xi),l)\varphi_l'(\xi)\,d\xi \psi_l^{(m)}(y).
$$
Thus, similarly to \eqref{2.15} the following estimate is easily obtained: uniformly with respect to $t\in\mathbb R$
$$
\|I_1\|_{L_{2}(\Sigma)} \leq c\|\mu\|_{L_2(B)}.
$$
Next, 
\begin{multline*}
I_2=  \sum\limits_{l=1}^{l_0} \frac1{2\pi}
\int_{|\theta|>2((b-\lambda_l)/3)^{3/2} } (i\theta)^j r_0^n(\theta,l)e^{it\theta}e^{r_0(\theta,l)x}\widehat\mu(\theta,l)\,d\theta \psi_l^{(m)}(y) \\+
 \sum\limits_{l=l_0+1}^{+\infty} \frac1{2\pi}
\int_{\mathbb R} (i\theta)^j r_0^n(\theta,l)e^{it\theta}e^{r_0(\theta,l)x}\widehat\mu(\theta,l)\,d\theta \psi_l^{(m)}(y).
\end{multline*}
We use the following fundamental inequality from \cite{BSZ}:
if certain continuous function $\gamma(\kappa)$ satisfies an inequality 
$\Re \gamma(\kappa)\leq -\varepsilon|\kappa|$ for some $\varepsilon>0$ and
all $\kappa\in \mathbb R$, then
$$
\Bigl\|\int_{\mathbb R} e^{\gamma(\kappa)x} f(\kappa)\,d\kappa\Bigr\|_
{L_2(\mathbb R^x_+)} \leq c(\varepsilon)\|f\|_{L_2(\mathbb R)}.
$$
Changing variables $\theta=\varphi_l(\kappa)$ we derive 
with the use of \eqref{2.28}--\eqref{2.30} that uniformly with respect to $t\in \mathbb R$ for  $\gamma(\kappa)=-c_0(|\kappa|-2\sqrt{(b-\lambda_l)/3})$, $\chi(\theta,l)=1$ for $|\theta|>2((b-\lambda_l)/3)^{3/2})$ (then $|\kappa|>2\sqrt{(b-\lambda_l)/3}$) and $\chi(\theta,l)=0$ for other values of $\theta$ if $l\leq l_0$, $\gamma(\kappa)=-c_0|\kappa|$, $\chi(\theta,l)\equiv 1$ if $l>l_0$
\begin{multline*}
\|I_2\|_{L_{2,+}} \\=  \frac1{2\pi}\Bigl(\sum\limits_{l=1}^{+\infty} 
\Bigl\|\int_{\mathbb R} \theta^j  r_0^n(\theta,l)e^{it\theta}e^{r_0(\theta,l)x}\widehat\mu(\theta,l)\chi(\theta,l)
\,d\theta 
\Bigr\|_{L_2(\mathbb R_+^x)}^2\Bigl\|\psi^{(m)}_l\Bigr\|_{L_2(0,L)}^2\Bigr)^{1/2} \\ \leq 
c\Bigl(\sum\limits_{l=1}^{+\infty}  
\Bigl\|\int_{\mathbb R} |\theta|^j (|\theta|^{1/3}+\lambda_l^{1/2})^n e^{p(\theta,l)x}|\widehat\mu(\theta,l)| 
\chi(\theta,l)\,d\theta \Bigr\|_{L_2(\mathbb R_+^x)}^2 l^{2m}\Bigr)^{1/2}\\ \leq
c_1\Bigl(\sum\limits_{l=1}^{+\infty}  
\Bigl\|\int_{\mathbb R} (\kappa^2+l^2)^{(3j+n)/2} e^{\gamma(\kappa)x}|\widehat\mu(\varphi_l(\kappa),l)| \chi(\varphi_l(\kappa),l)\varphi'_l(\kappa) 
\,d\kappa \Bigr\|_{L_2(\mathbb R_+^x)}^2 l^{2m}\Bigr)^{1/2}\\ \leq
c_2 \Bigl(\sum\limits_{l=1}^{+\infty} \left\|(\kappa^2+l^2)^{(3j+n)/2}
\widehat\mu(\varphi_l(\kappa),l)\chi(\varphi_l(\kappa),l)\varphi'_l(\kappa)\right\|_{L_2(\mathbb R^\kappa)}^2 l^{2m}\Bigr)^{1/2} \\ \leq 
c_3 \Bigl(\sum\limits_{l=1}^{+\infty}\bigl\|(|\theta|^{2/3}+l^2)^{(3j+n+m+1)/2}
\widehat\mu(\theta,l)\bigr\|_{L_2(\mathbb R^\theta)}^2 \Bigr)^{1/2}
= c_3 \|\mu\|_{\widetilde H^{(s+1)/3,s+1}(B)}.
\end{multline*}
Finally, use interpolation.
\end{proof}

\begin{lemma}\label{L2.8}
Let $\mu\in \widetilde H^{s/3,s}(B)$ for certain $s\geq -1/2$. Then for any $T>0$ and $j\leq s/3+1/6$
\begin{equation}\label{2.37}
\|\partial_t^j J(\cdot,\cdot,\cdot;\mu)\|_{L_2(0,T;\widetilde H^{s-3j+1/2}_+)} \leq
c(T,s)\|\mu\|_{\widetilde H^{s/3,s}(B)}.
\end{equation}
\end{lemma}

\begin{proof}
Without loss of generality one can assume that $\mu \in\widetilde{\EuScript S}(\overline{B})$. By virtue of \eqref{2.29}, \eqref{2.30} there exists $l_0$ such that for $l>l_0$ and all $\theta$ and there exists $\theta_0\geq 1$ such that for $|\theta|\geq \theta_0$ and all $l$
\begin{equation}\label{2.38}
p(\theta,l) \leq -c_0(|\theta|^{1/3}+l)
\end{equation}
Similarly to \eqref{2.18} divide $\mu$ into two parts:
\begin{equation}\label{2.39}
\mu_{0}(t,y)\equiv \sum\limits_{l=1}^{l_0} \EuScript F^{-1}_t\bigr[ \widehat \mu(\theta,l)\eta(\theta_0+1-|\theta|)\bigr](x)\psi_l(y),\quad
\mu_{1}(t,y)\equiv \mu(t,y)-\mu_0(t,y).
\end{equation}

Let $s+1/2$ be an integer, then if $3j+n+m=s+1/2$ we derive from equality \eqref{2.36} and inequalities \eqref{2.28}, \eqref{2.38} that for the obviously defined function $\chi(\theta,l)$ (in particular, $\chi(\theta,l)=0$ for $l\leq l_0$, $|\theta|\leq \theta_0$)
\begin{multline}\label{2.40}
\|\partial_t^j\partial_x^n\partial_y^m J(\cdot,\cdot,\cdot;\mu_1)\|_{L_2(\mathbb R^t\times\Sigma_+)} \\ =
\Bigl(\sum\limits_{l=1}^{+\infty} 
\Bigl\| \theta^j  r_0^n(\theta,l)\Bigl(\int_{\mathbb R_+}e^{-2p(\theta,l)x}dx\Bigr)^{1/2}\widehat\mu(\theta,l) \chi(\theta,l)
\Bigr\|_{L_2(\mathbb R^\theta)}^2\Bigl\|\psi^{(m)}_l\Bigr\|_{L_2(0,L)}^2\Bigr)^{1/2} \\ \leq
c\Bigl(\sum\limits_{l=1}^{+\infty}  
\Bigl\| (|\theta|^j(|\theta|^{2/3}+l^2)^{(n-1)/2} |\widehat\mu(\theta,l)| 
\,d\theta \Bigr\|_{L_2(\mathbb R_+^x)}^2 l^{2m}\Bigr)^{1/2} \leq
c_1\|\mu\|_{\widetilde H^{s/3,s}(B)}.
\end{multline}
For $\mu_0$ inequality \eqref{2.35} yields that for any $s_0\geq 0$
\begin{multline}\label{2.41}
\|\partial_t^j J(\cdot,\cdot,\cdot;\mu_0)\|_{L_2(0,T;\widetilde H_+^{s_0})} \leq T^{1/2}
\|\partial_t^j J(\cdot,\cdot,\cdot;\mu_0)\|_{C_b(\mathbb R^t;\widetilde H_+^{s_0})}  \\ \leq 
cT^{1/2}\|\mu_0\|_{\widetilde H^{(s_0+1)/3+j, s_0+1+3j}(B)} \leq 
c_1T^{1/2}\|\mu\|_{\widetilde H^{s/3, s}(B)}. 
\end{multline}
To finish the proof we again use interpolation.
\end{proof}

\begin{corollary}\label{C2.1}
Let $\mu\in \widetilde H^{s/3,s}(B)$ for certain $s>1/2$. Then for any $T>0$ and $j,k$, such that $3j+k< s-1/2$,
\begin{equation}\label{2.42}
\|\partial_t^j J(\cdot,\cdot,\cdot;\mu)\|_{L_2(0,T;W_{\infty,+}^k)} \leq
c(T,s)\|\mu\|_{\widetilde H^{s/3,s}(B)}.
\end{equation}
\end{corollary}

\begin{proof}
Estimate \eqref{2.42} obviously follows from \eqref{2.37} and the well-known embedding $H^{k+1+\varepsilon}_+\subset W_{\infty,+}^k$.
\end{proof}

\begin{lemma}\label{L2.9}
Let $\mu\in \widetilde H^{s/3,s}(B)$ for certain $s\in\mathbb R$. Then the function $J(t,x,y;\mu)$ is infinitely differentiable for $x>0$, $(t,y)\in \overline{B}$ and satisfies equation \eqref{2.1}, where $f\equiv 0$. Moreover, for any $T>0$, $x_0>0$ and $j, n$
\begin{equation}\label{2.43}
\sup\limits_{x\geq x_0} \|\partial_x^n J(\cdot,x,\cdot;\mu)\|_{\widetilde H^{j, 3j}(B_T)} \leq c(T,x_0,n,j,s) 
\|\mu\|_{\widetilde H^{s/3,s}(B)}.
\end{equation} 
\end{lemma}

\begin{proof}
Without loss of generality one can assume that $\mu \in\widetilde{\EuScript S}(\overline{B})$.
By virtue of \eqref{2.32}
$$
\partial_x^n J(t,x,y;\mu) \equiv \sum\limits_{l=1}^{+\infty} \EuScript F^{-1}_t \Bigl[r_0^n(\theta,l) e^{r_0(\theta,l)x}\widehat\mu(\theta,l)\Bigr](t) \psi_l(y).
$$
Again divide $\mu$ into two parts as in \eqref{2.39}. Then by virtue of \eqref{2.28} and \eqref{2.38}
\begin{multline*}
\|\partial_x^n J(\cdot,x,\cdot;\mu_1)\|_{\widetilde H^{j, 3j}(B)}^2  \\ \leq 
c\sum\limits_{l=1}^{+\infty}\|(|\theta|^{2/3}+l^2)^{(n+3j)/2} e^{-c_0x_0(|\theta|^{1/3}+l)} \widehat\mu(\theta,l)\|^2_{L_2(\mathbb R^\theta)} \leq c(x_0)\|\mu\|_{\widetilde H^{s/3,s}(B)}^2.
\end{multline*}
For $\mu_0$ apply estimate \eqref{2.35} similarly to \eqref{2.41}.

Equality \eqref{2.1} for $u\equiv J$, $f\equiv 0$ follows from \eqref{2.26}, \eqref{2.27}.
\end{proof}

\begin{lemma}\label{L2.10}
Let $\mu\in L_2(B)$ and $\mu(t,y)=0$ for $t<0$, then the function $J(t,x,y;\mu)$ for any $T>0$ is a weak solution (from $L_2(\Pi_T^+)$) to problem \eqref{2.1} (for $f\equiv 0$), \eqref{2.2} (for $u_0\equiv 0$), \eqref{1.3}, \eqref{1.4}.
\end{lemma}

\begin{proof}
First let $\mu\in C_0^\infty(B_+)$. Consider the smooth solution $u(t,x,y)$to the considered problem constructed in Lemma~\ref{L2.2}. For any $p=\varepsilon+i\theta$, 
where $\varepsilon>0$, define the Laplace--Fourier transform-coefficients
$$
\widetilde u(p,x,l) \equiv \int_{\mathbb R_+}\!\! \int_0^L e^{-pt} \psi_l(y) u(t,x,y)\,dydt.
$$
The function $\widetilde u(p,x,l)$ solves a problem
\begin{gather*}
p\widetilde u(p,x,l)+b\widetilde u_x(p,x,l)+\widetilde u_{xxx}(p,x,l)-\lambda_l \widetilde u_x(p,x,l)=0,\\
\widetilde u(p,0,l)=\widetilde \mu(p,l)\equiv \int_{\mathbb R_+}\!\! \int_0^L
e^{-pt}\psi_l(y)\mu(t,y)\,dydt,
\end{gather*}
whence, since $\widetilde u(p,x,l)\to 0$ as $x\to +\infty$, it follows, that
$$
\widetilde u(p,x,l)=\widetilde \mu(p,l)e^{z_0(p,l)x},
$$
where $z_0(p,l)$ is defined in \eqref{2.26}.
Using the formula of inversion of the Laplace transform we find, that the Fourier coefficients of the function $u$ are the following:
$$
\widehat u(t,x,l) = e^{\varepsilon t} \EuScript F_t^{-1}\left[
\widetilde \mu(\varepsilon+i\theta,l)e^{z_0(\varepsilon+i\theta,l)x}\right](t)
$$
and, therefore,
$$
u(t,x,y)= \sum\limits_{l=1}^{+\infty} e^{\varepsilon t} \EuScript F_t^{-1}\left[
\widetilde \mu(\varepsilon+i\theta,l)e^{z_0(\varepsilon+i\theta,l)x}\right](t)\psi_l(y).
$$
Passing to the limit as $\varepsilon\to+0$, we derive that $u(t,x,y) \equiv J(t,x,y;\mu)$.

In the general case approximate the function $\mu$ by smooth ones, pass to the limit on the basis of estimate \eqref{2.37} for $s=0$  (note, that this estimate is superfluous, the corresponding more weak estimate in $L_2(\Pi_T^+)$ is sufficient) and use the uniqueness result. 
\end{proof}

\begin{corollary}\label{C2.2}
Let $u_0\in L_{2,+}$, $\mu\in \widetilde H^{1/3,1}(B_T)$, $f\in L_2(0,T;L_{2,+})$ for certain $T>0$. Then there exists a unique solution to problem \eqref{2.1}, \eqref{1.2}--\eqref{1.4}, such that $u\in C([0,T];L_{2,+})$, $u,u_x \in C_b(\overline{\mathbb R}_+;L_2(B_T))$, given by a formula
\begin{equation}\label{2.44}
u(t,x,y)= S(t,x,y;u_0) +K(t,x,y;f) +J(t,x,y;\mu-S(0,\cdot,\cdot;u_0)-K(0,\cdot,\cdot;f)),
\end{equation}
where for the construction of the functions $S$ and $K$ the functions $u_0$ and $f$ are extended somehow in the same classes for $x<0$ and for the construction of the function $J$ the function $\mu-S(0,\cdot,\cdot;u_0)-K(0,\cdot,\cdot;f)$ is extended by zero for $t<0$ and somehow in the same class as $\mu$ for $t>T$.
\end{corollary}

\begin{proof}
This assertion directly succeeds from \eqref{2.15}, \eqref{2.16}, \eqref{2.22}, \eqref{2.23}, \eqref{2.25}, \eqref{2.34}, \eqref{2.35} and Lemma~\ref{L2.10}.
\end{proof}

We introduce certain additional function space. Let $\widetilde{\EuScript S}_{exp}(\overline{\Sigma}_+)$ denotes a space of infinitely smooth functions $\varphi(x,y)$ in $\overline{\Sigma}_{+}$,  such that $e^{n x}|\partial^\alpha\varphi(x,y)|\leq c(n,\alpha)$ for any $n$, multi-index $\alpha$, $(x,y)\in \overline{\Sigma}_+$ and $\partial_y^{2m}\varphi\big|_{y=0} =\partial_y^{2m}\varphi\big|_{y=L}=0$ in the case a), $\partial_y^{2m+1}\varphi\big|_{y=0} =\partial_y^{2m+1}\varphi\big|_{y=L}=0$ in the case b), $\partial_y^{2m}\varphi\big|_{y=0} =\partial_y^{2m+1}\varphi\big|_{y=L}=0$ in the case c), $\partial_y^{m}\varphi\big|_{y=0} =\partial_y^{m}\varphi\big|_{y=L}=0$ in the case d) for any $m$.

Let $\widetilde \Phi_0(x,y) \equiv u_0(x,y)$
and for $j\geq 1$
$$
\widetilde\Phi_j(x,y) \equiv \partial^{j-1}_t f(0,x,y)-
(b\partial_x+\partial_x^3+\partial_x\partial_y^2)\widetilde\Phi_{j-1}(x,y). 
$$

\begin{lemma}\label{L2.11}
Let $u_0\in \widetilde{\EuScript S}(\overline{\Sigma})\cap \widetilde{\EuScript S}_{exp}(\overline{\Sigma}_+)$, $f\in C^\infty\bigl([0,2T]; \widetilde{\EuScript S}(\overline{\Sigma})\cap\widetilde{\EuScript S}_{exp}(\overline{\Sigma}_+)\bigr)$, $\mu\in \widetilde{\EuScript S}(\overline{B}_+)$ and $\partial_t^j \mu(0,y) \equiv \widetilde\Phi_j(0,y)$ for any $j$.
Then there exists a unique solution to problem \eqref{2.1}, \eqref{1.2}--\eqref{1.4} $u\in C^\infty\bigl([0,T]; \widetilde{\EuScript S}_{exp}(\overline{\Sigma}_+)\bigr)$.
\end{lemma}

\begin{proof}
Let $w(t,x,y)=S(t,x,y;u_0) +K(t,x,y;f)$ be the solution to initial-boundary value problem \eqref{2.1}, \eqref{1.2}, \eqref{1.4} from the space $u\in C^\infty\bigl([0,2T];\widetilde{\EuScript S}(\overline{\Sigma})\cap\widetilde{\EuScript S}_{exp}(\overline{\Sigma}_+)\bigr)$ (see, for example, \cite{BF13}).

Let $\widetilde\mu(t,y)\equiv\bigl(\mu(t,y) -w(t,0,y)\bigr)\eta(2-t/T)$. Extend this function to the whole strip $B$ by zero for $t< 0$. Such an extension can be performed by virtue of the compatibility conditions on the line $t=0, x=0$. Then $\widetilde\mu\in \widetilde{\EuScript S}(\overline{B})$.

Then formula \eqref{2.44} provides the solution to the considered problem such that $\partial_t^j u\in  C([0,T],\widetilde H^k_+)$ for all $j$ and $k$ (see Lemma~\ref{L2.7}). 

Finally, let $v(t,x,y)\equiv u(t,x,y)\eta(x-1)$. The function $v$ solves an initial value problem in a strip $\Sigma$ of \eqref{2.1}, \eqref{1.2} type, where $f$, $u_0$ are substituted by by corresponding functions $F$, $v_0$ from the same classes and \cite{BF13} provides that $v\in C^\infty\bigl([0,T];\widetilde{\EuScript S}_{exp}(\overline{\Sigma}_+)\bigr)$.
\end{proof}

\begin{remark}\label{R2.3}
In further lemmas of this section all intermediate argument is performed for smooth solutions constructed in Lemma~\ref{L2.11} with consequent pass to the limit on the basis of obtained estimates due to linearity of the problem. 
\end{remark}

\begin{lemma}\label{L2.12}
Let $\mu\equiv 0$, $\rho(x)$ be an admissible weight function, such that $\rho'(x)$ is also an admissible weight function, $u_0\in L_{2,+}^{\rho(x)}$, $f\equiv f_0 +f_{1x} +f_2$, where   $f_0\in L_1(0,T;L_{2,+}^{\rho(x)})$, $f_1\in L_2(0,T;L_{2,+}^{\rho^2(x)/\rho'(x)})$, $f_2\rho^{3/4}(x)/(\rho'(x))^{1/4} \in L_{4/3}(\Pi_T^+)$. Then there exists a (unique) weak solution to problem \eqref{2.1}, \eqref{1.2}--\eqref{1.4} from the space $X^{\rho(x)}(\Pi_T^+)$ and a function $\nu\in L_2(B_T)$, such that for any function $\phi\in L_2(0,T;\widetilde H_+^2)$, $\phi_t, \phi_{xxx}, \phi_{xyy}\in L_2(\Pi_T^+)$, $\phi\big|_{t=T}=0$, $\phi\big|_{x=0} =0$, the following equality holds:
\begin{multline}\label{2.45}
\iiint_{\Pi_T^+}\Bigl[u(\phi_t+b\phi_x+\phi_{xxx}+\phi_{xyy}) +(f_0+f_2)\phi -f_1\phi_x\Bigr]\,dxdydt \\
+\iint u_0\phi\big|_{t=0}\,dxdy -
\iint_{B_T} \nu\phi_x\big|_{x=0}\,dydt =0.
\end{multline}
Moreover, for $t\in (0,T]$
\begin{multline}\label{2.46}
\|u\|_{X^{\rho(x)}(\Pi_t^+)} +\|\nu\|_{L_2(B_t)}\leq c(T)\Bigl(\|u_0\|_{L_{2,+}^{\rho(x)}} +\|f_0\|_{L_1(0,t;L_{2,+}^{\rho(x)})} \\ +\|f_1\|_{L_2(0,t;L_{2,+}^{\rho^2(x)/\rho'(x)})} +\|f_2\rho^{3/4}(x)/(\rho'(x))^{1/4}\|_{L_{4/3}(\Pi_t^+)}\Bigr),
\end{multline}
\begin{multline}\label{2.47}
\iint u^2(t,x,y)\rho(x)\,dxdy + \int_0^t \!\! \iint (3u_x^2 +u_y^2 -bu^2)\rho'(x)\,dxdyd\tau + \rho(0)\iint_{B_t} \nu^2\,dyd\tau \\ =
\iint u_0^2\rho(x)\,dxdy +\int_0^t \!\! \iint u^2\rho'''(x)\,dxdyd\tau +
2\int_0^t \!\! \iint (f_0+f_2)u\rho(x)\,dxdyd\tau  \\
-2\int_0^t \!\! \iint f_1\bigl(u\rho(x)\bigr)_x\,dxdyd\tau.
\end{multline}
If $f_1=f_2\equiv 0$, then in equality \eqref{2.47} one can put $\rho\equiv 1$.
\end{lemma}

\begin{proof}
Multiplying \eqref{2.1} by $2u(t,x,y)\rho(x)$ and integrating over $\Sigma_+$ we find that
\begin{multline}\label{2.48}
\frac d{dt} \iint u^2\rho\,dxdy +\rho(0)\int_0^L u_{x}^2\big|_{x=0}\,dy +
\iint (3u_{x}+u_{y}^2 -bu^2)\rho' \,dxdy \\ = \iint u^2 \rho'''\,dxdy + 
2\iint (f_0+f_2)u\rho\,dxdy -2\iint f_1(u\rho)_x\,dxdy.
\end{multline}
Note that
\begin{multline}\label{2.49}
\Bigl|\iint f_2u\rho\,dxdy\Bigr| \leq \|u(\rho'\rho)^{1/4}\|_{L_{4,+}} 
\|f_2\rho^{3/4}(\rho')^{-1/4}\|_{L_{4/3,+}} \\ \leq
c\Bigl[ \||Du|(\rho')^{1/2}\|_{L_{2,+}}^{1/2}\|u\rho^{1/2}\|_{L_{2,+}}^{1/2} +
\|u\rho^{1/2}\|_{L_{2,+}}\Bigr]\|f_2\rho^{3/4}(\rho')^{-1/4}\|_{L_{4/3,+}} \\ \leq
\varepsilon\iint |Du|^2\rho'\,dxdy +
c(\varepsilon)\|f_2\rho^{3/4}(\rho')^{-1/4}\|_{L_{4/3,+}}^{4/3}
\Bigl(\iint u^2\rho\,dxdy\Bigr)^{1/3} \\ +
c\|f_2\rho^{3/4}(\rho')^{-1/4}\|_{L_{4/3,+}}\Bigl(\iint u^2\rho\,dxdy\Bigr)^{1/2},
\end{multline}
\begin{multline}\label{2.50}
\Bigl|\iint f_1(u\rho)_x\,dxdy\Bigr| \leq c\|f_1\rho(\rho')^{-1/2}\|_{L_{2,+}}
\|(|u_{x}|+|u|)(\rho')^{1/2}\|_{L_{2,+}}\\ \leq \varepsilon
\iint \bigl(u_{x}^2+u^2\bigr)\rho'\,dxdy +
c(\varepsilon)\|f_1\|_{L_{2,+}^{\rho^2(x)/\rho'(x)}}^2,
\end{multline}
where $\varepsilon>0$ can be chosen arbitrarily small.
Equality \eqref{2.48} and inequalities \eqref{2.49}, \eqref{2.50} imply that that for smooth solutions
$$
\|u\|_{X^{\rho(x)}(\Pi_T^+)} + \|u_{x}\big|_{x=0}\|_{L_2(B_T)} \leq c.
$$

The end of the proof is standard.
\end{proof}

\begin{remark}\label{R2.4}
The method of construction of weak solution in Lemma~\ref{L2.13} via closure ensures that $u|_{x=0}=0$ in the trace sense (this fact can be also easily derived from equality \eqref{2.45}, since $u_x \in L_2((0,T)\times (0,x_0)\times (0,L))$ for certain $x_0>0$). Moreover, if it is known, in addition, that $u_x \in C_w([0,x_0];L_2(B_T))$ for certain $x_0>0$, then equality \eqref{2.45} yields that $u_x|_{x=0}=\nu$ (for example, one can put $\varphi \equiv x\eta(1-x/h)\omega(t,y)$ for $h>0$ and any $\omega\in C_0^\infty(B_T)$ and then tend $h$ to zero).
\end{remark}

\begin{lemma}\label{L2.13}
Let $\mu\equiv 0$, $\rho(x)$ be an admissible weight function, such that $\rho'(x)$ is also an admissible weight function, $u_0\in \widetilde H^{1,{\rho(x)}}_+$, $u_0\big|_{x=0}\equiv 0$, $f\equiv f_0+f_1$, where $f_0\in L_2(0,T;\widetilde H^{1,{\rho(x)}}_+)$, $f_1\in L_2(0,T;L_{2,+}^{\rho^2(x)/\rho'(x)})$. Then there exists a (unique) weak solution to problem \eqref{2.1}, \eqref{1.2}--\eqref{1.4}  from the space $X^{1,\rho(x)}(\Pi_T^+)$.
Moreover, for any $t\in (0,T]$
\begin{equation}\label{2.51}
\|u\|_{X^{1,\rho(x)}(\Pi_t^+)} 
\leq c(T) \Bigl(\|u_0\|_{\widetilde H^{1,\rho(x)}_+}+\|f_0\|_{L_2(0,t;\widetilde H^{1,\rho(x)}_+)}
+\|f_1\|_{L_2(0,t;L_{2,+}^{\rho^2(x)/\rho'(x)})}\Bigr),
\end{equation}
and for any smooth positive function $\gamma(t)$
\begin{multline}\label{2.52}
\gamma(t)\iint(u_x^2+u_y^2)\rho \,dxdy
+\int_0^t \gamma(\tau) \iint(3u_{xx}^2+4u^2_{xy}+u^2_{yy})\rho' \,dxdyd\tau \\ \leq
\gamma(0)\iint(u_{0x}^2+u_{0y}^2)\rho\,dxdy +
\int_0^t \gamma'(\tau) \iint(u_{x}^2+u^2_{y})\rho\,dxdyd\tau \\
+\int_0^t \gamma\iint(u^2_x+u^2_y)\rho'''\,dxdyd\tau 
+2\int_0^t \gamma\iint(f_{0x}u_x+f_{0y}u_y)\rho \,dxdyd\tau  \\+
\int_{B_t} \gamma f_0^2\big|_{x=0}\,dyd\tau
-2\int_0^t\gamma\iint f_1[(u_x\rho)_x+u_{yy}\rho]\,dxdyd\tau + \\
+b\int_0^t\gamma\iint (u_x^2+u_y^2)\rho'\,dxdyd\tau
+c(\rho,b)\iint_{B_t} \gamma u^2_x\big|_{x=0}\,dyd\tau.
\end{multline}
\end{lemma}

\begin{proof}
In the smooth case multiplying \eqref{2.1} by $-2\bigl((u_x(t,x,y)\rho(x)\bigr)_x
+$\newline $u_{yy}(t,x,y)\rho(x)\bigr)$ and integrating over $\Sigma_+$, one obtains an equality:
\begin{multline}\label{2.53}
\frac{d}{dt}\iint(u_x^2+u_y^2)\rho \,dxdy
+\iint(3u_{xx}^2+4u^2_{xy}+u^2_{yy})\rho' \,dxdy - b\iint(u_x^2+u_y^2)\rho'\,dxdy\\ =
 \int_0^L (u_{xx}^2\rho + u_{xy}^2\rho+2u_{xx}u_x\rho'-u_x^2\rho''-bu_x^2\rho)\big|_{x=0}\,dy\\
+\iint(u^2_x+u^2_y)\rho'''\,dxdy 
+2\iint(f_{0x}u_x+f_{0y}u_y)\rho \,dxdy \\+
\int_0^L (f_0u_x\rho)\big|_{x=0}\,dy
-2\iint f_1[(u_x\rho)_x+u_{yy}\rho]\,dxdy.
\end{multline}
Since the trace of $u_x$ on the plane $x=0$ is already estimated in \eqref{2.46} (here $\nu=u_x|_{x=0}$, see Remark~\ref{R2.4}) equality \eqref{2.53} provides that
$$
\|u\|_{X^{1,\rho(x)}(\Pi_T^+)} \leq c.
$$
\end{proof}

\begin{lemma}\label{L2.14}
Let the hypothesis of Lemma~\ref{L2.13} be satisfied in the case $\rho(x)\equiv e^{2\alpha x}$ for certain $\alpha>0$.  Consider the weak solution $u\in X^{1,\rho(x)}(\Pi_T^+)$ to problem \eqref{2.1}, \eqref{1.2}--\eqref{1.4}. Then for any $t\in (0,T]$ the following equality holds: 
\begin{multline}\label{2.54}
-\frac13\iint u^3(t,x,y)\rho(x)\,dxdy 
+\frac b3\int_0^t \!\!\iint u^3\rho'\,dxdyd\tau \\
+2\int_0^t\!\iint(u_{xx}+u_{yy}) uu_x\rho \,dxdyd\tau 
+\int_0^t\!\iint (u_{xx}+u_{yy}) u^2\rho' \,dxdyd\tau \\
=-\frac13\iint u_0^3\rho \,dxdy-\int_0^t\!\iint fu^2\rho \,dxdyd\tau.
\end{multline}
\end{lemma}

\begin{proof}. In the smooth case multiplying \eqref{2.1} by $-u^2(t,x,y)\rho(x)$ and integrating one instantly obtains equality \eqref{2.54}.

In the general case this equality is established via closure. Note that by virtue of \eqref{1.15} (for $q=4$,
$\rho_1=\rho_2\equiv\rho$) if $u\in X^{1,\rho(x)}(\Pi_T^+)$ then 
$$
u\in C([0,T];L_{4,+}^{\rho(x)}),\quad |Du|\in L_2(0,T; L_{4,+}^{\rho(x)})
$$
and this passage to the limit is easily justified.
\end{proof}

\begin{lemma}\label{L2.15}
Let $\mu\equiv 0$, $\rho(x)$ be an admissible weight function, such that $\rho'(x)$ is also an admissible weight function, $u_0\in \widetilde H^{2,{\rho(x)}}_+$, $u_0\big|_{x=0}\equiv 0$ and $u_{0xxx},u_{0xyy} \in L^{\rho(x)}_{2,+}$, $f\in C([0,T];L_{2,+}^{\rho(x)})$, moreover, $f\equiv f_0+f_{1x}$, where $f_0,f_{0t}\in L_1(0,T;L_{2,+}^{\rho(x)})$, $f_1,f_{1t}\in L_2(0,T;L_{2,+}^{\rho^2(x)/\rho'(x)})$. Then for the (unique) weak solution to problem \eqref{2.1}, \eqref{1.2}--\eqref{1.4} from the space $X^{\rho(x)}(\Pi_T^+)$  there exists $u_t\in X^{\rho(x)}(\Pi_T^+)$, which is the weak solution to problem of \eqref{2.1}, \eqref{1.2}--\eqref{1.4} type, where $f$ is substituted by $f_t$, $u_0$ -- by $\bigl(f\big|_{t=0}-bu_{0x} -u_{0xxx}- u_{0xyy}\bigr)$, $\mu\equiv 0$.
\end{lemma}

\begin{proof}
The proof for the function $v\equiv u_t$ is similar to Lemma~\ref{L2.12}.
\end{proof}

\begin{lemma}\label{L2.16}
Let the hypotheses of Lemma~\ref{L2.13} and Lemma~\ref{L2.15} be satisfied and, in addition, $f\in L_1(0,T;\widetilde H_+^{(0,2),\rho(x)})$. Then there exists a (unique) solution to problem \eqref{2.1}, \eqref{1.2}--\eqref{1.4} from the space $X^{2,\rho(x)}(\Pi_T^+)$ and for any $t\in [0,T]$
\begin{multline}\label{2.55}
\|u\|^2_{X^{2,\rho(x)}(\Pi_t^+)} \leq c(T)\Bigl(\|u_{0yy}\|^2_{L_{2,+}^{\rho(x)}} +
\|f\|^2_{C([0,t];L_{2,+}^{\rho(x)})} + \|u_t\|^2_{C([0,t];L_{2,+}^{\rho(x)})} \\+
\|u\|^2_{C([0,t];\widetilde H_+^{1,\rho(x)})} + 
\sup\limits_{\tau\in (0,t]}\Bigl|\int_0^{\tau} \!\!\iint f_{yy}u_{yy}\rho\,dxdyds\Bigr| +
\int_0^t \!\!\iint (u_{xx}^2+u_{yy}^2)\rho\,dxdyd\tau \\+
\int_0^t \!\!\iint (f^2+u_t^2+b^2u_x^2+bu_{yy}^2)\rho'\,dxdyd\tau\Bigr).
\end{multline}
\end{lemma}

\begin{proof}
For smooth solutions differentiating equality \eqref{2.1} twice with respect to $y$, multiplying the obtained equality by $2u_{yy}(t,x,y)\rho(x)$ and integrating over $\Sigma_+$ we derive similarly to  \eqref{2.48} that
\begin{multline}\label{2.56}
\frac d{dt} \iint u_{yy}^2\rho\,dxdy +\rho(0)\int_0^L u_{xyy}^2\big|_{x=0}\,dy +
\iint (3u_{xyy}^2+u_{yyy}^2-bu^2_{yy})\rho' \,dxdy \\ = \iint u_{yy}^2 \rho'''\,dxdy + 
2\iint f_{yy}u_{yy}\rho\,dxdy,
\end{multline}
whence obviously follows that
\begin{equation}\label{2.57}
\|u_{yy}\|_{X^{\rho(x)}(\Pi_T^+)} \leq c.
\end{equation}
Hence, for the weak solution also $u_{yy}\in X^{\rho(x)}(\Pi_T^+)$. Lemmas~\ref{L2.13} and~\ref{L2.15} provide that $u\in X^{1,\rho(x)}(\Pi_T^+)$, $u_{t}\in X^{\rho(x)}(\Pi_T^+)$. Write equality \eqref{2.1} in the form
\begin{equation}\label{2.58}
u_{xxx} = f-u_t -bu_x -u_{xyy}.
\end{equation}
Then, inequality \eqref{1.16} for $j=2$ and \eqref{2.58} yield that 
\begin{multline}\label{2.59}
\|u_{xx}\|_{L_{2,+}^{\rho(x)}} \leq c(\rho)\bigl(\|u_{xxx}\|_{H^{(-1,0),\rho(x)}} + 
\|u\|_{\widetilde H^{1,\rho(x)}_+}\bigr) \\ \leq
c(\rho,b)\bigl(\|f\|_{L^{\rho(x)}_{2,+}} + 
\|u_t\|_{L^{\rho(x)}_{2,+}} + \|u_{yy}\|_{L^{\rho(x)}_{2,+}} + \|u\|_{\widetilde H^{1,\rho(x)}_+}\bigr).
\end{multline}
Since
$$
\iint u_{xy}^2\rho\,dxdy = \iint u_{xx}u_{yy}\rho\,dxdy +\iint u_{yy}u_x\rho'\,dxdy,
$$
estimates \eqref{2.57} and \eqref{2.59} yield that $u\in C([0,T];\widetilde H_+^{2,\rho(x)})$ and
\begin{equation}\label{2.60}
\|u(t,\cdot,\cdot)\|_{\widetilde H^{2,\rho(x)}_+} \leq c\bigl( \|f\|_{L^{\rho(x)}_{2,+}} + 
\|u_t\|_{L^{\rho(x)}_{2,+}} + \|u_{yy}\|_{L^{\rho(x)}_{2,+}} + \|u\|_{\widetilde H^{1,\rho(x)}_+}\bigr).
\end{equation}
Next, 
\begin{multline*}
\iint u_{xxy}^2\rho'\,dxdy = \iint u_{xxx}u_{xyy}\rho'\,dxdy +\iint u_{xyy}u_{xx}\rho''\,dxdy \\
+\int_0^L (u_{xyy}u_{xx}\rho')\big|_{x=0}\,dy
\end{multline*}
and inequality \eqref{1.17} provides that
\begin{equation}\label{2.61}
\iint u_{xxy}^2\rho'\,dxdy \leq \iint (u_{xxx}^2+u_{xyy}^2)\rho'\,dxdy + \int_0^L u^2_{xyy}\big|_{x=0}\,dy +
c\iint u_{xx}^2\rho\,dxdy.
\end{equation}
From equality \eqref{2.58} we derive that
\begin{equation}\label{2.62}
\iint u_{xxx}^2\rho'\,dxdy \leq c\iint (f^2+u_t^2+b^2u_x^2+u_{xyy}^2)\rho'\,dxdy
\end{equation}
and combining \eqref{2.56}, \eqref{2.60}--\eqref{2.62} finish the proof.
\end{proof}

\begin{lemma}\label{L2.17}
Let $\mu\equiv 0$, $\rho(x)$ be an admissible weight function, such that $\rho'(x)$ is also an admissible weight function, $u_0\in \widetilde H^{3,\rho(x)}_+$, $u_0(0,y)\equiv 0$, $f\in C([0,T];L_{2,+}^{\rho(x)})$ and $f\equiv f_0+f_{1x}$, where $f_0\in L_2(0,T;\widetilde H_+^{1,\rho(x)})\cap L_1(0,T;\widetilde H_+^{(0,3),\rho(x)})$, $f_{0t}\in L_1(0,T;L_{2,+}^{\rho(x)})$, $f_1\in L_2(0,T;L_{2,+}^{\rho^2(x)/\rho'(x)})$, $f_{1x}\in L_2(0,T;\widetilde H_+^{(0,2),\rho^2(x)/\rho'(x)})$, $f_{1t}\in L_2(0,T;L_{2,+}^{\rho^2(x)/\rho'(x)})$. Then there exists a (unique) solution to problem \eqref{2.1}, \eqref{1.2}--\eqref{1.4} from the space $X^{3,\rho(x)}(\Pi_T^+)$ and for any $t\in (0,T]$
\begin{multline}\label{2.63}
\|u\|_{X^{3,\rho(x)}(\Pi_t^+)} \leq c(T)\bigl(\|u_0\|_{\widetilde H_+^{3,\rho(x)}} +
\|f\|_{C([0,t];L_{2,+}^{\rho(x)})} + \|f_0\|_{L_2(0,t;\widetilde H_+^{1,\rho(x)})} \\+
\|f_{0yyy}\|_{L_1(0,t;L_{2,+}^{\rho(x)})} + \|f_{0t}\|_{L_1(0,t;L_{2,+}^{\rho(x)})} +
\|f_1\|_{L_2(0,t;L_{2,+}^{\rho^2(x)/\rho'(x)})} \\ 
+ \|f_{1x}\|_{L_2(0,t;L_{2,+}^{\rho^2(x)/\rho'(x)})} +
\|f_{1xyy}\|_{L_2(0,t;L_{2,+}^{\rho^2(x)/\rho'(x)})} +\|f_{1t}\|_{L_2(0,t;L_{2,+}^{\rho^2(x)/\rho'(x)})}\bigr).
\end{multline} 
\end{lemma}

\begin{proof}
First of all note that hypotheses of Lemmas~\ref{L2.12} (for $f_2\equiv 0$), \ref{L2.13}, \ref{L2.15} and~\ref{L2.16} are satisfied. Therefore, taking into account also Remark~\ref{R2.4} we derive for smooth solutions that
\begin{equation}\label{2.64}
\|u\|_{X^{2,\rho(x)}(\Pi_T^+)} + \|u_x\big|_{x=0}\|_{L_2(B_T)} + 
\|u_t\|_{X^{\rho(x)}(\Pi_T^+)} + \|u_{tx}\big|_{x=0}\|_{L_2(B_T)}  \leq c.
\end{equation}
Next, differentiating equality \eqref{2.1} twice with respect to $y$, multiplying the obtained equality by $-2u_{yyyy}(t,x,y)\rho(x)$ and integrating over $\Sigma_+$ we derive similarly to  \eqref{2.56} that
\begin{multline}\label{2.65}
\frac d{dt} \iint u_{yyy}^2\rho\,dxdy +\rho(0)\int_0^L u_{xyyy}^2\big|_{x=0}\,dy +
\iint (3u_{xyyy}+u_{yyyy}^2-bu^2_{yyy})\rho' \,dxdy \\ = \iint u_{yyy}^2 \rho'''\,dxdy + 
2\iint f_{0yyy}u_{yyy}\rho\,dxdy - 2\iint f_{1yy}u_{yyyy}\rho\,dxdy.
\end{multline}
Here
$$
\Bigl|2\iint f_{1yy}u_{yyyy}\rho\,dxdy \Bigr| \leq \varepsilon \iint u_{yyyy}^2\rho'\,dxdy +
\frac 1\varepsilon\iint f_{1yy}^2 \frac{\rho^2}{\rho'}\,dxdy,
$$
where $\varepsilon>0$ can be chosen arbitrarily small, and equality \eqref{2.65} yields that
\begin{equation}\label{2.66}
\|u_{yyy}\|_{X^{\rho(x)}(\Pi_T^+)} + \|u_{xyyy}\big|_{x=0}\|_{L_2(B_T)}  \leq c.
\end{equation}
Again apply equality \eqref{2.58}. Then it follows from \eqref{2.66} that we have the suitable estimate on $u_y$ in the space $L_2(0,T; \widetilde H_+^{3,\rho'(x)})$ and, in particular, on $u_{xxyy}$ in $L_2(0,T;L_{2,+}^{\rho'(x)})$ (for similar argument see \eqref{2.61}). One more application of \eqref{2.58} yields the estimate on $u_{xxxx}$ in $L_2(0,T;L_{2,+}^{\rho'(x)})$. As a result
\begin{equation}\label{2.67}
\|u\|_{L_2(0,T;\widetilde H_+^{4,\rho'(x)})} \leq c.
\end{equation}

Consider the extensions of the functions $u$ and $f$ for $y\in (L,2L]$ and $y\in [-L,0)$ in the case a) by the  even reflections through $y=L$ and $y=0$, in the case b) -- by the odd ones, in the case c) -- by the corresponding combination of these methods, in the case d) -- by the periodic extension. Then the functions $u$ and $f$ remain smooth in the more wide domain $[0,T]\times \overline{\mathbb R}_+ \times [-L,2L]$, and equality \eqref{2.1} also remains valid. Let $\eta_L(y) \equiv \eta(1+y/L)\eta(2-y/L)$, $\widetilde u(t,x,y) \equiv u(t,x,y)\eta_L(y)$, $\widetilde f(t,x,y) \equiv f(t,x,y)\eta_L(y)$. 
Now we apply the inequality (see, e.g. \cite{LM}) for the domain $\mathbb R^2_+ = \{(x,y); x>0\}$
$$
\|g\|_{H^2(\mathbb R^2_+)}\leq c\bigl(\|\Delta g\|_{L_2(\mathbb R^2_+)}+
\|g\big|_{x=0}\|_{H^{3/2}(\mathbb R)}+
\|g\|_{H^1(\mathbb R^2_+)}\bigr)
$$
for the function $g\equiv \widetilde u_x\rho^{1/2}(x)$. Note that 
$$
\Delta_{x,y} (\widetilde u_x\rho^{1/2}) = (\widetilde f -\widetilde u_t -b\widetilde u_x +2u_{xy}\eta_L'+u_x\eta_L'')\rho^{1/2} +
2\widetilde u_{xx}(\rho^{1/2})' +\widetilde u_x(\rho^{1/2})''.
$$
It follows from (\ref{2.64}) that
$$
\|\Delta_{x,y} (\widetilde u_x\rho^{1/2})\|_{C([0,T];L_2(\mathbb R^2_+)} \leq c.
$$
Moreover, by virtue of (\ref{2.64}), (\ref{2.66}) and embedding
$H^2(\mathbb R^2_+)\subset H^{3/2}(\{{x=0}\}\times \mathbb R)$ (see \cite{LM})
\begin{multline*}
\|u_x\big|_{x=0}\|_{C([0,T];H^{3/2}(\mathbb R))} \leq 
\|u_{0x}\big|_{x=0}\|_{H^{3/2}(\mathbb R)}\\
+2\|u_{tx}\big|_{x=0}\|_{L_2((0,T)\times \mathbb R)}^{1/2}
\|u_x\big|_{x=0}\|_{L_2(0,T;H^3(\mathbb R))}^{1/2} \leq c.
\end{multline*}
Therefore,
\begin{equation}\label{2.68}
\|u_x\|_{C([0,T];H_+^{2,\rho(x)})} \leq c.
\end{equation}
Estimates \eqref{2.64}, \eqref{2.66}--\eqref{2.68} provide the desired result.
\end{proof}

\section{Existence of solutions}\label{S3}

Consider an auxiliary equation
\begin{equation}\label{3.1}
u_t+bu_x+u_{xxx}+u_{xyy}+(g(u))_x+(\psi(t,x,y)u)_x=f(t,x,y).
\end{equation}
The notion of a weak solution to problem \eqref{3.1}, \eqref{1.2}--\eqref{1.4} is similar to Definition~\ref{D1.1}.

\begin{lemma}\label{L3.1}
Let $g\in C^1(\mathbb R)$, $g(0)=0$, $|g'(u)|\leq c\ \forall u\in\mathbb R$, $\psi\in L_2(0,T;L_{\infty,+})$, $u_0\in L_{2,+}$, $f\in L_1(0,T;L_{2,+})$ and $u_0(x,y)= f(t,x,y)=0$ if $x>R$ for certain $R>0$, $\mu\equiv 0$. Then  problem \eqref{3.1}, \eqref{1.2}--\eqref{1.4} has a unique weak solution $u(t,x,y)$, such that $u\in X^{\rho(x)}(\Pi_T^+)$ for any $\alpha>0$ and $\rho(x)\equiv e^{2\alpha x}$.
\end{lemma}

\begin{proof}
We apply the contraction principle. Fix $\alpha>0$ and $\rho(x)\equiv e^{2\alpha x}$. For $t_0\in(0,T]$ define a mapping $\Lambda$ on $X^{\rho(x)}(\Pi_{t_0}^+)$ as follows: $u=\Lambda v\in X^{\rho(x)}(\Pi_{t_0}^+)$ is a weak solution to a linear problem
\begin{equation}\label{3.2}
u_t+bu_x+u_{xxx}+u_{xyy} =f-g(v)_x -(\psi v)_x
\end{equation}
in $\Pi_{t_0}$ with initial and boundary conditions \eqref{1.2}--\eqref{1.4}.

Note that $\rho^2(\rho')^{-1/2} \sim \rho$, $|g(v)|\leq c|v|$ and, therefore,
\begin{gather*}
\|g(v)\|_{L_2(0,t_0;L_{2,+}^{\rho^2(x)/\rho'(x)})}\leq 
c\|v\|_{C([0,t_0];L_{2,+}^{\rho(x)})}<\infty, \\
\|\psi v\|_{L_2(0,t_0;L_{2,+}^{\rho^2(x)/\rho'(x)})}\leq c\|\psi\|_{L_2(0,t_0;L_{\infty,+})}
\|v\|_{C([0,t_0];L_{2,+}^{\rho(x)})}<\infty.
\end{gather*}
Thus, Lemma~\ref{L2.12} provides that the mapping $\Lambda$ exists. Moreover, for functions $v,\widetilde{v}\in X^{\rho(x)}(\Pi_{t_0}^+)$
\begin{equation*}
\|g(v)-g(\widetilde{v})\|_{L_2(0,t_0;L_{2,+}^{\rho^2(x)/\rho'(x)})}\leq c\|v-\widetilde{v}\|_{L_2(0,t_0;L_{2,+}^{\rho(x)})}  \leq ct_0^{1/2}\|v-\widetilde{v}\|_{C([0,t_0];L_{2,+}^{\rho(x)})},
\end{equation*}
$$
\|\psi(v-\widetilde{v})\|_{L_2(0,t_0;L_2^{\rho^2(x)/\rho'(x)})} 
\leq c\|\psi\|_{L_2(0,t_0;L_{\infty,+})} \|v-\widetilde{v}\|_{C([0,t_0];L_{2,+}^{\rho(x)})}.
$$
As a result, according to inequality \eqref{2.46} (where $f_2\equiv 0$)
\begin{equation*}
\|\Lambda v-\Lambda\widetilde{v}\|_{X^{\rho(x)}(\Pi_{t_0}^+)}\leq
c(T)\omega(t_0)\|v-\widetilde{v}\|_{X^{\rho(x)}(\Pi_{t_0}^+)},
\end{equation*}
where $\omega(t_0)\to 0$ as $t_0\to +0$ and $\omega$ depends on the properties of continuity of the primitives of the function $\|\psi(t,\cdot,\cdot)\|_{L_{\infty,+}}^2$ on $[0,T]$.
Since the constant in the right side of this inequality is uniform with respect to $u_0$ and $f$, one can construct the solution on the whole time segment $[0,T]$ by the standard argument.
\end{proof}

Now we pass to the results of existence in Theorem~\ref{T1.1}.

\begin{proof}[Proof of Existence Part of Theorem~\ref{T1.1}]
First of all we make zero boundary data for $x=0$. Let
\begin{equation}\label{3.3}
\psi(t,x,y)\equiv J(t,x,y;\mu)\eta(2-x),
\end{equation}
where for the construction of the boundary potential $J$ the function $\mu$ is extended to the whole strip $B$ in the same class. Then the results of Section~\ref{S2} provide that
\begin{equation}\label{3.4}
\left\{
\begin{aligned}
&\widetilde\psi \equiv\psi_t +b\psi_x+\psi_{xxx}+\psi_{xyy} \in C^\infty(\overline{\Pi}_T^+), \quad
\psi=0 \quad \text{for}\ x\geq 2,\\
&\psi\in C([0,T];L_{2,+})\cap L_2(0,T;\widetilde H_+^1)\cap L_2(0,T;W^1_{\infty,+}),\\
&\psi \in C([0,2];\widetilde H^{1/3,1}(B_T)), \quad \psi\big|_{x=0}=\mu,\quad
\psi_x \in C([0,2]; L_2(B_T)).
\end{aligned}
\right.
\end{equation}
Consider a function
\begin{equation}\label{3.5}
U(t,x,y) \equiv u(t,x,y) -\psi(t,x,y).
\end{equation}
Then $u\in X_w^{\rho(x)}(\Pi_T^+)$ is a weak solution to problem \eqref{1.1}--\eqref{1.4} iff $U\in X_w^{\rho(x)}(\Pi_T^+)$ is a weak solution to an initial-boundary value problem in $\Pi_T^+$ for an equation
\begin{equation}\label{3.6}
U_t+bU_x+U_{xxx}+U_{xyy}+UU_x+(\psi U)_x=F\equiv f-\widetilde\psi- \psi\psi_x,
\end{equation}
with initial and boundary conditions
\begin{equation}\label{3.7}
U\big|_{t=0} = U_0\equiv u_0 -\psi\big|_{t=0},\quad U\big|_{x=0}=0
\end{equation}
and the same boundary conditions on $\Omega_{T,+}$ as \eqref{1.4}. Note also that the functions $U_0$, $F$ satisfy the same assumptions as the corresponding functions $u_0$, $f$ in the hypothesis of the theorem.
 
For $h\in (0,1]$ consider a set of initial-boundary value problems in $\Pi_T^+$
\begin{equation}\label{3.8}
U_t+bU_x+U_{xxx}+U_{xyy}+\left(g_h(U)\right)_x+(\psi U)_x = F_h
\end{equation}
with boundary conditions \eqref{1.4} and
\begin{equation}\label{3.9}
U\big|_{t=0} = U_{0h},\quad U\big|_{x=0}=0,
\end{equation}
where 
\begin{equation}\label{3.10}
F_h(t,x,y) \equiv F(t,x,y)\eta(1/h-x), \quad U_{0h}(x,y)\equiv U_0\eta(1/h-x)
\end{equation}
and
\begin{equation*}
g_h(u)\equiv\int_0^u\Bigl[\theta\eta(2-h|\theta|)+\frac{2\sgn\theta}{h}\eta(h|\theta|-1)\Bigr]\,d\theta.
\end{equation*}
Note that $g_h(u)= u^2/2$ if $|u|\leq 1/h$, $|g'_h(u)|\leq 2/h\ \forall u\in\mathbb R$ and $|g'_h(u)|\leq 2|u|$ uniformly with respect to $h$.

According to Lemma~\ref{L3.1}, there exists a unique solution to this problem $U_h\in X^{e^{2\alpha x}}(\Pi_T^+)$ for any $\alpha>0$.

Next, establish appropriate estimates for functions $U_h$ uniform with respect to~$h$ (we drop the index $h$ in intermediate steps for simplicity). First, note that $g'(U)U_x, \psi U_x, \psi_x U, F \in L_1(0,T;L_{2,+}^{\rho(x)})$ and so the hypothesis of Lemma~\ref{L2.12} is satisfied (for $f_1=f_2\equiv 0$). Write down the analogue of equality \eqref{2.47} for $\rho\equiv 1$, then:
\begin{equation}\label{3.11}
\iint U^2 \,dxdy \leq \iint U_0^2\,dxdy 
+\int_0^t \!\!\iint \bigl(2F- 2(g(U))_x-\psi_x U \bigr)U\,dxdyd\tau.
\end{equation}
Since
\begin{equation}\label{3.12}
(g(U))_xU=\partial_x\Bigl(\int_0^U g'(\theta)\theta \,d\theta\Bigr) 
\end{equation}
we derive that
\begin{equation}\label{3.13}
\iint (g(U))_x U \,dxdy = 0.
\end{equation}
Therefore, uniformly with respect to $h$ (and also uniformly with respect to $L$)
\begin{equation}\label{3.14}
\|u_h\|_{C([0,T];L_{2,+})} \leq c.
\end{equation}

Next, equalities \eqref{2.47} and \eqref{3.12} provide that
\begin{multline}\label{3.15}
\iint U^2\rho \,dxdy + \int_0^t\!\! \iint (3U_x^2 +U_y^2)\rho'\,dxdyd\tau 
\leq \iint U_0^2\rho\,dxdy  \\ +c \int_0^t \!\!\iint U^2\rho\, dxdyd\tau
+2\int_0^t \!\!\iint FU\rho\,dxdyd\tau \\ 
+\int_0^t \!\!\iint (\psi\rho'-\psi_x\rho)U^2\,dxdyd\tau
+2\int_0^t \!\!\iint \Bigl(\int_0^U g'(\theta)\theta \,d\theta \Bigr) \rho' \,dxdyd\tau.
\end{multline}
Note that 
\begin{equation}\label{3.16}
\Bigl|\int_0^U g'(\theta)\theta \,d\theta \Bigr| \leq c|U|^3.
\end{equation}
Applying interpolating inequality \eqref{1.15} for $\rho_1=\rho_2\equiv\rho'$, we obtain that
\begin{multline}\label{3.17}
\iint |U|^3\rho'\,dxdy \leq \Bigl(\iint U^2\,dxdy \iint U^4(\rho')^2\,dxdy\Bigr)^{1/2}
\leq c\Bigl(\iint U^2\,dxdy\Bigr)^{1/2} \\ \times \Bigl[\Bigl(\iint |DU|^2\rho'dxdy\Bigr)^{1/2}
\Bigl(\iint U^2\rho' \,dxdy\Bigr)^{1/2}
+\iint U^2\rho' \,dxdy\Bigr]
\end{multline}
(note that here the constant $c$ is also uniform with respect to $L$ in the cases a) and c)).
Since the norm of the functions $u_h$ in the space $L_2$ is already  estimated in \eqref{3.14}, it follows from \eqref{3.15}--\eqref{3.17} that uniformly with respect to $h$
\begin{equation}\label{3.18}
\|u_h\|_{X^{\rho(x)}(\Pi_T^+)} \leq c.
\end{equation}

Finally, write down the analogue of \eqref{3.15}, where $\rho(x)$ is substituted by \linebreak $\rho_0(x-x_0)$ for any $x_0\geq 0$. Then it easily follows that (see \eqref{1.9})
\begin{equation}\label{3.19}
\lambda^+ (|Du_h|;T)\leq c.
\end{equation}

From equation \eqref{3.8} itself, estimate \eqref{3.14} and the well-known embedding $L_{1,+}\subset H^{-2}_+$, it follows that uniformly with respect to $h$
\begin{equation}\label{3.20}
\|u_{ht}\|_{L_1(0,T;H^{-3}_+)}\leq c.
\end{equation}

Estimates \eqref{3.18}--\eqref{3.20} by the standard argument provide existence of a weak solution to problem \eqref{1.1}--\eqref{1.4} $u\in X^{\rho(x)}(\Pi_T)$, $\lambda^+(|Du|;T) <\infty$ (see, for example, \cite{F12}) as a limit of functions $u_h$ when
$h\to +0$.
\end{proof}

We now proceed to solutions in spaces $\widetilde H^{1,\rho(x)}$ and first estimate a lemma analogous to Lemma~\ref{L3.1}.

\begin{lemma}\label{L3.2}
Let $g(u)\equiv u^2/2$, $\psi\in L_2(0,T;W^1_{\infty,+})\cap L_2(0,T;\widetilde H_+^2)$, $u_0\in \widetilde H^1_+$, $u_0\big|_{x=0}\equiv 0$,  $f\in L_2(0,T;\widetilde H_+^1)$ and $u_0(x,y)= f(t,x,y)=\psi(t,x,y)=0$ if $x>R$ for certain $R>0$, $\mu\equiv 0$. Then  problem \eqref{3.1}, \eqref{1.2}--\eqref{1.4} has a unique weak solution $u(t,x,y)$, such that $u\in X^{1,\rho(x)}(\Pi_T^+)$ for any $\alpha>0$ and $\rho(x)\equiv e^{2\alpha x}$.
\end{lemma}

\begin{proof}
Fix $\alpha>0$ and $\rho(x)\equiv e^{2\alpha x}$. For $t_0\in(0,T]$ define a mapping $\Lambda$ on $X^{1,\rho(x)}(\Pi_{t_0}^+)$ as follows: $u=\Lambda v\in X^{1,\rho(x)}(\Pi_{t_0}^+)$ is a weak solution to a linear problem
\begin{equation}\label{3.21}
u_t+bu_x+u_{xxx}+u_{xyy} =f-vv_x -(\psi v)_x 
\end{equation}
in $\Pi_{t_0}$ with initial and boundary conditions \eqref{1.2}--\eqref{1.4}.

Note that by virtue of \eqref{1.15} for $\rho_1=\rho_2\equiv \rho\geq 1$
\begin{multline}\label{3.22}
\|vv_x\|_{L_2(0,t_0;L_{2,+}^{\rho(x)})} \leq
\Bigl[\int_0^{t_0}\|v_x\rho^{1/2}\|_{L_{4,+}}^2 \|v\rho^{1/2}\|_{L_{4,+}}^2\,dt\Bigr]^{1/2} \\ \leq c\Bigl[\int_0^{t_0} \Bigl(\bigl\||D v_x|\bigr\|_{L_{2,+}^{\rho(x)}}\|v_x\|_{L_{2,+}^{\rho(x)}}+
\|v_x\|_{L_{2,+}^{\rho(x)}}^{2}\Bigr) \bigl\| |Dv|+|v|\bigr\|_{L_{2,+}^{\rho(x)}}^{2}\,dt\Bigr]^{1/2} \\ \leq
c_1 t_0^{1/4} \|v\|_{L_2(0,t_0;\widetilde H_+^{2,\rho(x)})}^{1/2}\|v\|_{C([0,t_0];\widetilde H_+^{1,\rho(x)})}^{3/2} \leq
c_1 t_0^{1/4} \|v\|_{X^{1,\rho(x)}(\Pi_{t_0}^+)}^2
\end{multline}
and similarly
\begin{equation}\label{3.23}
\|vv_x-\widetilde v\widetilde v_x\|_{L_2(0,t_0;L_{2,+}^{\rho(x)})} \leq 
ct_0^{1/4}\bigl(\|v\|_{X^{1,\rho(x)}(\Pi^+_{t_0})}+\|\widetilde v\|_{X^{1,\rho(x)}(\Pi_{t_0}^+)}\bigr) 
\|v-\widetilde v\|_{X^{1,\rho(x)}(\Pi_{t_0}^+)},
\end{equation}
$$
\|(\psi v)_x\|_{L_2(0,t_0;L_{2,+}^{\rho(x)})} \leq c \bigl(\|\psi\|_{L_2(0,t_0;L_{\infty,+})}
+\|\psi_x\|_{L_2(0,t_0;L_{\infty,+})}\bigr)||v\|_{C([0,t_0];\widetilde H_+^{1,\rho(x)})}.
$$
In particular, the hypothesis of Lemma~\ref{L2.13} is satisfied (since $\rho^2(\rho')^{-1/2} \sim \rho$) and, therefore, the mapping $\Lambda$ exists.

Moreover, by virtue of \eqref{2.51}
\begin{multline}\label{3.24}
\|\Lambda v\|_{X^{1,\rho(x)}(\Pi_{t_0}^+)}\leq 
c(T)\Bigl(\|u_0\|_{\widetilde H^{1,\rho(x)}}+
 \|f\|_{L_1(0,T;\widetilde H^{1,\rho(x)})} +
\omega(t_0)\|v\|_{X^{1,\rho(x)}(\Pi_{t_0}^+)}  \\ 
+t_0^{1/4}\|v\|_{X^{1,\rho(x)}(\Pi_{t_0}^+)}^2\Bigr),
\end{multline}
\begin{multline}\label{3.25}
\|\Lambda v - \Lambda\widetilde v\|_{X^{1,\rho(x)}(\Pi_{t_0}^+)}\leq c(T) 
\Bigl(\omega(t_0)\|v-\widetilde v\|_{X^{1,\rho(x)}(\Pi_{t_0}^+)}  \\ 
+t_0^{1/4}\bigl(\|v\|_{X^{1,\rho(x)}(\Pi_{t_0}^+)}+
\|\widetilde v\|_{X^{1,\rho(x)}(\Pi_{t_0}^+)}\bigr)
\|v-\widetilde v\|_{X^{1,\rho(x)}(\Pi_{t_0}^+)} \Bigr),
\end{multline}
where $\omega(t_0)\to 0$ as $t_0\to +0$ and $\omega$ depends on the properties of continuity of the primitives of the functions $\|\psi(t,\cdot,\cdot)\|_{L_{\infty,+}}^2$ and 
$\|\psi_x(t,\cdot,\cdot)\|_{L_{\infty,+}}^2$ on $[0,T]$.

Existence of the unique weak solution to the considered problem in the space 
$X^{1,\rho(x)}(\Pi_{t_0}^+)$ on the time interval $[0,t_0]$, depending on 
$\|u_0\|_{\widetilde H_+^{1,\rho(x)}}$, follows from \eqref{3.24}, \eqref{3.25} by the standard argument.

In order to finish the proof, we establish the following a priori estimate: if $u\in X^{1,\rho(x)}(\Pi_{T'}^+)$ is a solution to the considered problem for some $T'\in (0,T]$ and $\psi(t,x,y)=0$ for $x>R$, then
\begin{multline}\label{3.26}
\|u\|_{X^{1,\rho(x)}(\Pi_{T'}^+)} \\ \leq c\bigl(\|u_0\|_{\widetilde H^{1,\rho(x)}},
\|f\|_{L_2(0,T;\widetilde H^{1,\rho(x)})}, \|\psi\|_{L_2(0,T;W^1_{\infty,+})}, 
\|\psi\|_{L_2(0,T;\widetilde H_+^2)}\bigr).
\end{multline}

First of all note that similarly to \eqref{3.14}, \eqref{3.18} one can derive from \eqref{2.47} that 
\begin{equation}\label{3.27}
\|u\|_{X^{\rho(x)}(\Pi_{T'}^+)} \leq c.
\end{equation}
Next, since the hypotheses of Lemma~\ref{L2.13} and, consequently, Lemma~\ref{L2.14} are satisfied, write down the corresponding analogues of inequality \eqref{2.52} for $\gamma\equiv 1$, equality \eqref{2.54} and sum them, then 
\begin{multline}\label{3.28}
\iint\bigl(u_x^2+u_y^2-\frac13 u^3\bigr)\rho \,dxdy
+\int_0^t \!\!\iint(3u_{xx}^2+4u^2_{xy}+u^2_{yy})\rho' \,dxdyd\tau \\ \leq
\iint\bigl(u_{0x}^2+u_{0y}^2-\frac{u_0^3}3\bigr)\rho\,dxdy 
+c\int_0^t \!\!\iint(u^2_x+u^2_y)\rho\,dxdyd\tau 
+2\int_0^t \!\!\iint uu_x^2\rho'\,dxdyd\tau \\
-\int_0^t \!\! \iint (u_{xx}+u_{yy})u^2\rho'\,dxdyd\tau
+\int_0^t \iint(2f_{x}u_x+2f_{y}u_y-fu^2)\rho \,dxdyd\tau  \\
+\int_{B_t}  f^2\big|_{x=0}\,dyd\tau
+c\iint_{B_t} u^2_x\big|_{x=0}\,dyd\tau
-\int_0^t \!\!\iint \Bigl(\frac{b}3 u^3+ \frac14 u^4\Bigr)\rho'\,dxdyd\tau \\
-\iint_{B_t} (\psi u_x^2\rho)\big|_{x=0}\,dyd\tau 
-\int_0^t \!\!\iint \psi_x(3u_x^2+u_y^2)\rho\,dxdyd\tau
+\int_0^t \!\!\iint  \psi(u_x^2+u_y^2)\rho'\,dxdyd\tau \\
-2\int_0^t \!\!\iint \psi_y u_xu_y\rho\,dxdyd\tau
-2\int_0^t \!\!\iint (\psi_{xx}uu_x +\psi_{xy}uu_y)\rho\,dxdyd\tau.
\end{multline}
Consider the integrals from \eqref{3.28}, where for the sake of the use in the sequel assume only that $\rho$ and $\rho'$ are admissible weight functions. Similarly to \eqref{3.17}
\begin{gather}\label{3.29}
\iint |u|^3\rho\,dxdy \leq c\Bigl(\iint |Du|^2\rho\,dxdy\Bigr)^{1/2}+c,\\ \label{3.30}
\iint u^4\rho^2\,dxdy \leq c\iint |Du|^2\rho\,dxdy+c,
\end{gather}
where the already obtained estimated \eqref{3.27} on $\|u\|_{C([0,T];L_{2,+}^{\rho(x)})}$ is also used. Next,
\begin{multline}\label{3.31}
\Bigl| \iint uu_x^2\rho'\,dxdy\Bigr|
\leq \Bigl(\iint u^2\,dxdy\Bigr)^{1/2}\Bigl(\iint u_x^4(\rho')^2\,dxdy\Bigr)^{1/2} \\
\leq c \Bigl[\Bigl(\iint |Du_x|^2\rho'\,dxdy\Bigr)^{1/2}\Bigl(\iint u_x^2\rho\,dxdy\Bigr)^{1/2}
+\iint u_x^2\rho\,dxdy\Bigr],
\end{multline}
\begin{equation}\label{3.32}
\Bigl|\iint (u_{xx}+u_{yy})u^2\rho'\,dxdy\Bigr|
\leq c\Bigl(\iint(u_{xx}^2+u_{yy}^2)\rho'\,dxdy\Bigr)^{1/2} \Bigl(\iint u^4\rho^2\,dxdy\Bigr)^{1/2}.
\end{equation}
Interpolating inequality \eqref{1.17}
provides that
\begin{multline}\label{3.33}
\int_0^L \bigl((1+|\psi|)u_x^2\bigr)\big|_{x=0}\,dy \leq \varepsilon \iint u_{xx}^2\rho'\,dxdy \\ +
c(\varepsilon)\bigl(1+\sup\limits_{(x,y)\in\Sigma_+} \psi^2\bigr) \iint u_x^2\rho\,dxdy,
\end{multline}
where $\varepsilon>0$ can be chosen arbitrarily small. Finally, since $\rho(x)\leq c(R)\rho'(x)$ for $x\in [0,R]$
\begin{multline}\label{3.34}
\Bigl|\iint \psi_{xx}uu_x\rho\,dxdy\Bigr|   \leq 
c\Bigl(\iint \psi_{xx}^2\,dxdy\Bigr)^{1/2} \Bigl(\iint u^4\rho^2\,dxdy
\iint u_x^4\rho'\rho\,dxdy\Bigr)^{1/4} \\ \leq
c_1\Bigl(\iint\! \psi_{xx}^2\,dxdy\Bigr)^{1/2} \Bigl[\Bigl(\iint |Du_x|^2\rho'\,dxdy \iint u_x^2\rho\,dxdy\Bigr)^{1/4}\!\!+\Bigl(\iint u_x^2\rho\,dxdy\Bigr)^{1/2}\Bigr] \\
\times  \Bigl[\Bigl(\iint |Du|^2\rho\,dxdy \Bigr)^{1/4}+1\Bigr] \leq
\varepsilon \iint |Du_x|^2\rho'\,dxdy  \\ +
c(\varepsilon) \Bigl[\iint \psi_{xx}^2\,dxdy+1\Bigr]
\Bigl( \iint |Du|^2\rho\,dxdy +1\Bigr).
\end{multline}
Other integrals in \eqref{3.28} are estimated in a obvious way and \eqref{3.26} follows.
\end{proof}

\begin{proof}[Proof of Existence Part of Theorem~\ref{T1.2}]
Introduce the function $\psi$ by formula \eqref{3.3}. Then in addition to properties \eqref{3.4} it follows from the results of Section~\ref{S2} that
\begin{equation}\label{3.35}
\left\{
\begin{aligned}
&\psi\in C([0,T];\widetilde H_+^1)\cap L_2(0,T;\widetilde H_+^2)\cap  C([0,2];\widetilde H^{2/3,2}(B_T)),\\
&\psi_x \in C([0,2];\widetilde H^{1/3,2}(B_T)), \quad \psi_{xx} \in C([0,2];L_2(B_T)).
\end{aligned}
\right.
\end{equation}
Again introduce the function $U(t,x,y)$ by formula \eqref{3.5} and consider problem \eqref{3.6}, \eqref{3.7}, \eqref{1.4} instead of \eqref{1.1}--\eqref{1.4}. Note that here \eqref{3.4}, \eqref{3.35} provide that that the properties of the functions  $U_0$, $F$, are the same as the corresponding ones for the functions $u_0$, $f$ in the hypothesis of the theorem.
 
For $h\in (0,1]$, consider a set of initial-boundary value problems in $\Pi_T^+$
\begin{equation}\label{3.36}
U_t+bU_x+U_{xxx}+U_{xyy}+UU_x+(\psi U)_x = F_h
\end{equation}
with boundary conditions \eqref{3.7}, \eqref{1.4}, where $F_h$ and $U_{0h}$ are given by \eqref{3.10}. 

Repeating the argument in \eqref{3.11}, \eqref{3.13}, \eqref{3.15}--\eqref{3.17} for $g(u)\equiv u^2/2$ and \eqref{3.28}--\eqref{3.34} we derive that uniformly with respect to $h$
\begin{equation}\label{3.37}
\|u_h\|_{X^{1,\rho(x)}(\Pi_T^+)} \leq c.
\end{equation}
Similarly to \eqref{3.19} one can obtain that 
\begin{equation}\label{3.38}
\lambda^+ (|D^2u_h|;T) \leq c.
\end{equation}
Estimates \eqref{3.37}, \eqref{3.38} and \eqref{3.20} provide existence of a weak solution to the considered problem $u\in X^{1,\rho(x)}(\Pi_T^+)$.
\end{proof}

Finally, consider regular solutions.

\begin{lemma}\label{L3.3}
Let $g(u)\equiv u^2/2$, $\mu\equiv 0$, the functions $u_0$ and $f$ satisfy the hypothesis of Theorem~\ref{T1.3}, $\psi\in X^{3,\rho(x)}(\Pi_T^+)\cap L_2(0,T;W_{\infty,+}^3)$, $\psi_t\in L_2(0,T;L_{\infty,+})$ and $\psi(t,x,y)=0$ if $x>R$ for certain $R>0$.  Then  problem \eqref{3.1}, \eqref{1.2}--\eqref{1.4} has a unique solution $u\in X^{3,\rho(x)}(\Pi_T^+)$.
\end{lemma}

\begin{proof}
For $t_0\in(0,T]$, $v\in X^{3,\rho(x)}(\Pi_{t_0}^+)$ let $u=\Lambda v\in X^{3,\rho(x)}(\Pi_{t_0}^+)$ be a solution to a linear problem \eqref{3.21}, \eqref{1.2}--\eqref{1.4}.

Apply Lemma~\ref{L2.17} where $f$ stands for $f_0$, $-(v^2/2+\psi v)$ -- for $f_1$. We have:
\begin{multline}\label{3.39}
\|vv_x+\psi v_x+\psi_x v\|_{C[0,t_0];L_{2,+}^{\rho(x)})} \leq
\|u_0u_{0x}+\psi\big|_{t=0}u_{0x}+\psi_x\big|_{t=0}u_0\|_{L_{2,+}^{\rho(x)}} \\ +
\|(vv_x)_t+(\psi v)_{tx}\|_{L_1(0,t_0;L_{2,+}^{\rho(x)})}
\end{multline}
and with the use of \eqref{1.18} derive that 
\begin{gather}\label{3.40}
\|u_0u_{0x}\|_{L_{2,+}^{\rho(x)}} \leq
c\|u_0\|_{L_{\infty,+}}\|u_{0x}\|_{L_{2,+}^{\rho(x)}} \leq
c_1\|u_0\|_{\widetilde H_+^{3,\rho(x)}}^2, \\
\label{3.41}
\|\psi\big|_{t=0}u_{0x}+\psi_x\big|_{t=0}u_0\|_{L_{2,+}^{\rho(x)}} \leq
c\bigl(\|\psi_x\big|_{t=0}\|_{L_{2,+}}+\|\psi\big|_{t=0}\|_{L_{\infty,+}}\bigr)
\|u_0\|_{\widetilde H_+^{3,\rho(x)}},
\end{gather}
since $\rho'(x)\geq 1$
\begin{multline}\label{3.42}
\|vv_{tx}\|_{L_1(0,t_0;L_{2,+}^{\rho(x)})} \leq 
\int_0^{t_0} \|v\rho^{1/2}\|_{L_{\infty,+}}\|v_{tx}\|_{L_{2,+}^{\rho'(x)}}\,dt \\ \leq
ct_0^{1/2}\|v\|_{X^{2,\rho(x)}(\Pi_{t_0}^+)}
\|v\|_{X^{3,\rho(x)}(\Pi_{t_0}^+)},
\end{multline}
\begin{multline}\label{3.43}
\|v_xv_{t}\|_{L_1(0,t_0;L_{2,+}^{\rho(x)})} \leq
\int_0^{t_0} \|v_x\rho^{1/2}\|_{L_{4,+}}\|v_t(\rho'\rho)^{1/4}\|_{L_{4,+}}\,dt \\ \leq
ct_0^{1/2}\|v\|_{X^{2,\rho(x)}(\Pi_{t_0}^+)}
\|v\|_{X^{3,\rho(x)}(\Pi_{t_0}^+)}
\end{multline}
and similarly
\begin{equation}\label{3.44}
\|(\psi v)_{tx}\|_{L_1(0,t_0;L_{2,+}^{\rho(x)})} \leq 
ct_0^{1/2}\|\psi\|_{X^{3,\rho(x)}(\Pi_{T}^+)}
\|v\|_{X^{3,\rho(x)}(\Pi_{t_0}^+)}.
\end{equation}
Next,
\begin{multline}\label{3.45}
\|vv_{t}\|_{L_2(0,t_0;L_{2,+}^{\rho^2(x)/\rho'(x)})} \leq 
\Bigl(\int_0^{t_0} \|v\rho^{1/2}\|^2_{L_{\infty,+}}\|v_{t}\|^2_{L_{2,+}^{\rho(x)}}\,dt\Bigr)^{1/2} \\ \leq
ct_0^{1/2}\|v\|_{X^{2,\rho(x)}(\Pi_{t_0}^+)}
\|v\|_{X^{3,\rho(x)}(\Pi_{t_0}^+)},
\end{multline}
$(vv_x)_{yy} = vv_{xyy} +2v_yv_{xy}+v_xv_{yy}$, where similarly to \eqref{3.45}
\begin{equation}\label{3.46}
\|vv_{xyy}\|_{L_2(0,t_0;L_{2,+}^{\rho^2(x)/\rho'(x)})} \leq 
ct_0^{1/2}\|v\|_{X^{2,\rho(x)}(\Pi_{t_0}^+)}
\|v\|_{X^{3,\rho(x)}(\Pi_{t_0}^+)},
\end{equation}
\begin{multline}\label{3.47}
\|v_yv_{xy}\|_{L_2(0,t_0;L_{2,+}^{\rho^2(x)/\rho'(x)})} \leq 
\Bigl(\int_0^{t_0}  \|v_y\rho^{1/2}\|^2_{L_{4,+}}\|v_{xy}\rho^{1/2}\|^2_{L_{4,+}}\,dt \Bigr)^{1/2} \\ \leq
ct_0^{1/2}\|v\|_{X^{2,\rho(x)}(\Pi_{t_0}^+)}
\|v\|_{X^{3,\rho(x)}(\Pi_{t_0}^+)}
\end{multline}
and similar estimate holds for $v_xv_{yy}$. Finally, similarly to \eqref{3.45}--\eqref{3.47}
\begin{multline}\label{3.48}
\|(\psi v)_t\|_{L_2(0,t_0;L_{2,+}^{\rho^2(x)/\rho'(x)})}+
\|(\psi v)_{xyy}\|_{L_2(0,t_0;L_{2,+}^{\rho^2(x)/\rho'(x)})} \\ \leq 
ct_0^{1/2}\|\psi\|_{X^{3,\rho(x)}(\Pi_{T}^+)}
\|v\|_{X^{3,\rho(x)}(\Pi_{t_0}^+)}.
\end{multline}
Moreover, the assumptions on the function $\psi$ ensure that the corresponding boundary conditions on the function $f_1$ are satisfied for $y=0$ and $y=L$. Therefore, the mapping $\Lambda$ exists and one can use estimate \eqref{2.63} to derive inequalities
\begin{equation}\label{3.49}
\|\Lambda v\|_{X^{3,\rho(x)}(\Pi_{t_0}^+)}\leq 
\widetilde c + c(T)t_0^{1/2}\bigl(\|\psi\|_{X^{3,\rho(x)}(\Pi_{T}^+)}\|v\|_{X^{3,\rho(x)}(\Pi_{t_0}^+)}  
+\|v\|_{X^{3,\rho(x)}(\Pi_{t_0}^+)}^2\bigr),
\end{equation}
\begin{multline}\label{3.50}
\|\Lambda v - \Lambda\widetilde v\|_{X^{3,\rho(x)}(\Pi_{t_0}^+)}\leq c(T)t_0^{1/2}
\Bigl(\|\psi\|_{X^{3,\rho(x)}(\Pi_{T}^+)}\|v-\widetilde v\|_{X^{3,\rho(x)}(\Pi_{t_0}^+)}  \\ 
+\bigl(\|v\|_{X^{3,\rho(x)}(\Pi_{t_0}^+)}+
\|\widetilde v\|_{X^{3,\rho(x)}(\Pi_{t_0}^+)}\bigr)
\|v-\widetilde v\|_{X^{3,\rho(x)}(\Pi_{t_0}^+)} \Bigr),
\end{multline}
where the constant $\widetilde c$ on the properties of functions $u_0$, $f$, $\psi$. Hence, existence of the unique solution to the considered problem in the space 
$X^{3,\rho(x)}(\Pi_{t_0}^+)$ on the time interval $[0,t_0]$, depending on 
$\|u_0\|_{\widetilde H_+^{3,\rho(x)}}$, follows by the standard argument.

Now establish the following a priori estimate: if $u\in X^{3,\rho(x)}(\Pi_{T'}^+)$ is a solution to the considered problem for some $T'\in (0,T]$, then
\begin{equation}\label{3.51}
\|u\|_{X^{3,\rho(x)}(\Pi_{T'}^+)}  \leq c,
\end{equation}
where the constant $c$ depends on $T$ and the properties of the functions $u_0$, $f$, $\psi$ from the hypothesis of the present lemma.

According to \eqref{3.26}
\begin{equation}\label{3.52}
\|u\|_{X^{1,\rho(x)}(\Pi_{T'}^+)} \leq c.
\end{equation}

Next, since the hypothesis of Lemma~\ref{L2.15} is fulfilled write down the corresponding analogue of equality \eqref{2.47} for the function $u_t$:
\begin{multline}\label{3.53}
\iint u_t^2\rho\,dxdy + \int_0^t \!\! \iint (3u_{tx}^2 +u_{ty}^2)\rho'\,dxdyd\tau \\ \leq
\iint (f-bu_x-u_{xxx}-u_{xyy}-uu_x-(\psi u)_x)^2\big|_{t=0}\rho\,dxdy +c\int_0^t \!\! \iint u_t^2\rho\,dxdyd\tau \\ + 2\int_0^t \!\! \iint (f-(\psi u)_x)_t u_t\rho\,dxdyd\tau  
+2\int_0^t \!\! \iint uu_t(u_t\rho)_x\,dxdyd\tau.
\end{multline}
Here since $\rho'\geq 1$ and estimate \eqref{3.52} holds
\begin{multline*}
2 \iint uu_t(u_t\rho)_x\,dxdy = \iint (u\rho'-u_x\rho)u_t^2\,dxdy  \\ \leq 
c\Bigl(\iint \bigl(u_x^2\frac{\rho}{\rho'}+u^2\bigr)\,dxdy \iint u_t^4\rho'\rho\,dxdy\Bigr)^{1/2} \\ \leq
c_1\Bigl[\Bigl(\iint |Du_t|^2\rho'\,dxdy \iint u_t^2\rho\,dxdy\Bigr)^{1/2} + \iint u_t^2\rho\,dxdy\Bigr],
\end{multline*}
\begin{multline*}
\Bigl|\iint \psi_{tx}uu_t\rho\,dxdy\Bigr| \\ \leq c(R)\Bigl(\iint \psi_{tx}^2\,dxdy\Bigr)^{1/2}
\Bigl(\iint u^4\rho^2\,dxdy \iint u_t^4\rho'\rho\,dxdy\Bigr)^{1/4} \\ \leq 
\varepsilon\iint |Du_t|^2\rho'\,dxdy + c(\varepsilon,R)\Bigl[\iint \psi_{tx}^2\,dxdy 
+1\Bigr] \Bigl( \iint u_t^2\rho\,dxdy+1\Bigr),
\end{multline*}
where $\varepsilon>0$ can be chosen arbitrarily small. Other terms in \eqref{3.53} are estimated in a obvious way and, consequently,
\begin{equation}\label{3.54}
\|u_t\|_{X^{\rho(x)}(\Pi_{T'}^+)} \leq c.
\end{equation}

Now apply Lemma~\ref{L2.16}, then inequality \eqref{2.55} and estimates \eqref{3.52}, \eqref{3.54} yield that for any $t\leq T'$
\begin{multline}\label{3.55}
\|u\|^2_{X^{2,\rho(x)}(\Pi_t^+)} \leq c+
c\|uu_x\|^2_{C([0,t];L_{2,+}^{\rho(x)})} \\ +
c \sup\limits_{\tau\in (0,t]}\Bigl|\int_0^\tau \!\! \iint \bigl(f-uu_x-(\psi u)_x\bigr)_{yy}u_{yy}\rho\,dxdyds\Bigr| 
+c\int_0^t \!\! \iint (u^2_{xx} +u^2_{yy})\rho\,dxdy
\end{multline}
We have
$$
\|uu_x\|^2_{L_{2,+}^{\rho(x)}} \leq c \|u_x\rho^{1/2}\|^2_{L_{4,+}} \leq
\varepsilon \bigl\| |Du_x| \bigr\|^2_{L_{2,+}^{\rho(x)}} +c(\varepsilon),
$$
where $\varepsilon>0$ can be chosen arbitrarily small;
\begin{equation*}
\iint (uu_x)_{yy}u_{yy}\rho\,dxdy = \frac12\iint (u_x\rho-u\rho')u^2_{yy}\,dxdy +
2\iint u_yu_{xy}u_{yy}\rho\,dxdy,
\end{equation*}
where again since $\rho'\geq 1$
\begin{multline*}
\Bigl|\iint u_yu_{xy}u_{yy}\rho\,dxdy\Bigr|  \leq  \Bigl(\iint u_y^2\rho\,dxdy \iint (u_{xy}^4+u_{yy}^4)\rho'\rho\,dxdy\Bigr)^{1/2} \\ \leq
c_1\Bigl(\iint |D^3 u|^2\rho'\,dxdy \iint |D^2 u|^2\rho\,dxdy\Bigr)^{1/2}.
\end{multline*}
Integral of $(u_x\rho-u\rho')u_{yy}^2$ is estimated in a similar way and it follows from \eqref{3.55} that
\begin{equation}\label{3.56}
\|u\|_{X^{2,\rho(x)}(\Pi_{T'}^+)} \leq c.
\end{equation}

Finally, apply Lemma~\ref{L2.17} on the basis of the already obtained estimates \eqref{3.54}, \eqref{3.56}, then inequality \eqref{2.63} and estimates \eqref{3.39}--\eqref{3.48} applied to $v\equiv u$  provide similarly to \eqref{3.49} that for any $t_0\in (0,T']$
\begin{equation*}
\|u\|_{X^{3,\rho(x)}(\Pi_{t_0}^+)}\leq 
\widetilde c + c(T)t_0^{1/2}\bigl(\|\psi\|_{X^{3,\rho(x)}(\Pi_{T}^+)}+\|u\|_{X^{2,\rho(x)}(\Pi_{T'}^+)}\bigr)
\|u\|_{X^{3,\rho(x)}(\Pi_{t_0}^+)},
\end{equation*}
whence \eqref{3.51} follows.
\end{proof}

\begin{proof}[Proof of Theorem~\ref{T1.3}]
Introduce the functions $\psi$, $U$ by formulas \eqref{3.3}, \eqref{3.5} and consider problem \eqref{3.6}, \eqref{3.7}, \eqref{1.4}. Then the functions $\psi$, $F\sim f$ and $U_0\sim u_0$ satisfy the hypothesis of Lemma~\ref{L3.3} and the result is immediate.
\end{proof}

\section{Uniqueness and continuous dependence}\label{S4}

\begin{theorem}\label{T4.1}
Let $\rho(x)$ be an admissible weight function, such that $\rho'(x)$ is also an admissible weight function and $\rho^{1/2}(x)\leq c_0\rho'(x) \ \forall x\geq 0$ for certain positive constant $c_0$. Then for any $T>0$ and $M>0$ there exist constant $c=c(T,M)$, such that for any two weak solutions $u(t,x,y)$ and $\widetilde u(t,x,y)$ to problem \eqref{1.1}--\eqref{1.4}, satisfying $\|u\|_{X_w^{\rho(x)}(\Pi_T^+)}, \|\widetilde u\|_{X_w^{\rho(x)}(\Pi_T^+)} \leq M$, with corresponding data $u_0, \widetilde u_0\in L_{2,+}^{\rho(x)}$, $\mu, \widetilde\mu\in \widetilde H^{1/3,1}(B_T)$, $f, \widetilde f\in L_1(0,T;L_{2,+}^{\rho(x)})$ the following inequality holds:
\begin{equation}\label{4.1}
\|u -\widetilde u\|_{X_w^{\rho(x)}(\Pi_T^+)} \leq c\bigl( \|u_0 - \widetilde u_0\|_{L_{2,+}^{\rho(x)}} +
\|\mu-\widetilde\mu\|_{H^{1/3,1}(B_T)} +
\|f-\widetilde f\|_{L_1(0,T;L_{2,+}^{\rho(x)})}\bigr).
\end{equation} 
\end{theorem}

\begin{proof}
Let the function $\psi$ is defined by formula \eqref{3.3}, the function $\widetilde\psi$ in a similar way for $\widetilde\mu$ and $\Psi\equiv \psi-\widetilde\psi$. Then, in particular,
\begin{equation}\label{4.2}
\|\Psi\|_{X^{\rho(x)}(\Pi_T^+)} \leq c(T)\|\mu-\widetilde\mu\|_{H^{1/3,1}(B_T)}.
\end{equation}
Let $U_0\equiv u_0-\widetilde u_0 - \Psi\big|_{t=0}$, $F\equiv f-\widetilde f -(\Psi_t+b\Psi_x+\Psi_{xxx}+\Psi_{xyy})$, then
\begin{gather}\label{4.3}
\|U_0\|_{L_{2,+}^{\rho(x)}} \leq \|u_0-\widetilde u_0\|_{L_{2,+}^{\rho(x)}} +
c(T)\|\mu-\widetilde\mu\|_{H^{1/3,1}(B_T)}, \\
\label{4.4}
\|F\|_{L_1(0,T;L_{2,+}^{\rho(x)})} \leq \|f-\widetilde f\|_{L_1(0,T;L_{2,+}^{\rho(x)})} +
c(T)\|\mu-\widetilde\mu\|_{H^{1/3,1}(B_T)}.
\end{gather}
The function $U(t,x,y) \equiv u(t,x,y) -\widetilde u(t,x,y) -\Psi(t,x,y)$ is a weak solution to an initial-boundary value problem in $\Pi_T^+$ for an equation
\begin{equation}\label{4.5}
U_t+bU_x+U_{xxx}+U_{xyy} = F -(uu_x-\widetilde u \widetilde u_x)
\end{equation}
with initial and boundary conditions \eqref{1.4},
\begin{equation}\label{4.6}
U\big|_{t=0} =U_0,\qquad U\big|_{x=0}=0.
\end{equation}
Apply Lemma~\ref{L2.12} where $f_2= -(uu_x-\widetilde u \widetilde u_x)$. Note that assumptions on the function $\rho$ provide that $\rho(\rho')^{-1/3} \leq c \rho^{1/3}\rho'$ and by virtue of \eqref{1.15}
\begin{multline*}
\|uu_x\rho^{3/4}(\rho')^{-1/4}\|_{L_{4/3}(\Pi_T^+)}^{4/3} \leq c 
\int_0^T \Bigl(\iint u^4\rho'\rho\,dxdy\Bigr)^{1/3} \Bigl(\iint u_x^2\rho'\,dxdy\Bigr)^{2/3}\,dt \\ \leq c_1
\|u\|_{L_\infty(0,T;L_{2,+}^{\rho(x)})}^{1/3} \|u\|_{L_2(0,T;H_+^{1,\rho'(x)})} <+\infty.
\end{multline*}
Therefore, we derive from \eqref{2.47} that for $t\in (0,T]$
\begin{multline}\label{4.7}
\iint U^2\rho\,dxdy +\int_0^t\!\! \iint (3U_x^2+U_y^2)\rho'\,dxdyd\tau \leq \iint U_0^2\rho\,dxdy  \\ 
+c\int_0^t\!\! \iint U^2\rho\,dxdyd\tau + 2\int_0^t \!\!\iint\bigl(F-(uu_x-\widetilde u \widetilde u_x)\bigr)U\rho\, dxdyd\tau.
\end{multline}
Here
\begin{equation}\label{4.8}
-2\iint (uu_x-\widetilde u \widetilde u_x)U\rho\,dxdy =
\iint (u+\widetilde u)(U+\Psi)(U\rho)_x\,dxdy.
\end{equation}
Then by virtue of \eqref{1.15} and the assumptions on the function $\rho$ (which yield that $(\rho/\rho')^3 \leq c \rho'\rho$)
\begin{multline}\label{4.9}
\iint |u(U+\Psi)U_x|\rho\,dxdy  \\ \leq c \Bigl(\iint u^4 (\rho/\rho')^3\,dxdy \iint (U^4+\Psi^4) \rho'\rho\,dxdy\Bigr)^{1/4} \Bigl(\iint U_x^2\rho'\,dxdy\Bigr)^{1/2} \\ \leq c_1 \Bigl(\|u\|^{1/2}_{H_+^{1,\rho'(x)}}\|u\|^{1/2}_{L_{2,+}^{\rho(x)}} +
\|u\|_{L_{2,+}^{\rho(x)}}\Bigr)
\Bigl[\Bigl(\iint \bigl(|DU|^2+|D\Psi|^2\bigr)\rho'\,dxdy\Bigr)^{3/4} \\ \times \Bigl(\iint (U^2+\Psi^2)\rho\,dxdy\Bigr)^{1/4}  +
\iint (U^2+\Psi^2)\rho\,dxdy\Bigr]
\end{multline}
and, therefore,
\begin{multline}\label{4.10}
\int_0^t\!\! \iint |u(U+\Psi)U_x|\rho\,dxdyd\tau \leq
\varepsilon \int_0^t\!\! \iint \bigl(|DU|^2+|D\Psi|^2\bigr)\rho'\,dxdyd\tau \\ + c(\varepsilon)
\int_0^t \gamma(\tau) \iint (U^2+\Psi^2)\rho\,dxdyd\tau,
\end{multline}
where $\varepsilon>0$ can be chosen arbitrarily small and $\gamma\equiv 1+\|u\|^2_{H_+^{1,\rho'(x)}}\in L_1(0,T)$. Then estimates \eqref{4.2}--\eqref{4.4}, \eqref{4.10} and inequality \eqref{4.7} provide the desired result.
\end{proof}

\begin{remark}\label{R4.1}
Theorems~\ref{T1.1} and \ref{T4.1} show that under the hypothesis of Theorem~\ref{T1.1} problem \eqref{1.1}--\eqref{1.4} is globally well-posed in the space $X_w^{\rho(x)}(\Pi_T^+)$.
\end{remark}

\begin{theorem}\label{T4.2}
Let $\rho(x)$ be an admissible weight function, such that $\rho'(x)$ is also an admissible weight function and $\rho'(x)\geq c_0 \ \forall x\geq 0$ for certain positive constant $c_0$. Then for any $T>0$ and $M>0$ there exist constant $c=c(T,M)$, such that for any two weak solutions $u(t,x,y)$ and $\widetilde u(t,x,y)$ to problem \eqref{1.1}--\eqref{1.4}, satisfying $\|u\|_{X_w^{1,\rho(x)}(\Pi_T^+)}, \|\widetilde u\|_{X_w^{1,\rho(x)}(\Pi_T^+)} \leq M$, with corresponding data $u_0, \widetilde u_0\in L_{2,+}^{\rho(x)}$, $\mu, \widetilde\mu\in \widetilde H^{1/3,1}(B_T)$, $f, \widetilde f\in L_1(0,T;L_{2,+}^{\rho(x)})$ inequality \eqref{4.1} holds.
\end{theorem}

\begin{proof}
The proof mostly repeats the proof of Theorem\ref{T4.1}. The difference is related only to the nonlinear term. Here we apply Lemma~\ref{L2.12} where $f_1= -(u^2-\widetilde u^2)/2$. Note that for any $t\in [0,T]$
$$
\|u^2\|_{L_{2,+}^{\rho^2(x)/\rho'(x)}} \leq \|u\|^2_{H_+^{1,\rho(x)}},
$$
in particular, $u^2 \in L_\infty(0,T;L_{2,+}^{\rho^2(x)/\rho'(x)})$. Write down inequality \eqref{4.7}. In comparison with \eqref{4.8} transform the integral of the nonlinear term in the following way:
\begin{multline*}
2\iint (uu_x-\widetilde u \widetilde u_x)U\rho\,dxdy = \frac12 \iint (u+\widetilde u)_xU^2\rho\,dxdy \\-
 \frac12 \iint (u+\widetilde u)U^2\rho'\,dxdy
+\iint \bigl((u+\widetilde u)\Psi\bigr)_x U\rho\,dxdy.
\end{multline*}
Here
\begin{multline*}
\iint |u_x|(U^2+\Psi^2)\rho\,dxdy \leq 
\Bigl(\iint u_x^2\frac{\rho}{\rho'}\,dxdy \iint (U^4+\Psi^4) \rho'\rho\,dxdy\Bigr)^{1/2} \\ \leq c\Bigl[\Bigl(\iint \bigl(|DU|^2+|D\Psi|^2\bigr)\rho'\,dxdy\iint (U^2+\Psi^2)\rho\,dxdy\Bigr)^{1/2} \!+
\iint (U^2+\Psi^2)\rho\,dxdy\Bigr],
\end{multline*}
\begin{multline*}
\iint |u\Psi_xU|\rho\,dxdy \leq c\Bigl(\iint u^4\rho^2\,dxdy \iint U^4\rho'\rho\,dxdy\Bigr)^{1/4}
\Bigl(\iint \Psi_x^2\rho\,dxdy\Bigr)^{1/2}  \\ \leq
\varepsilon \iint |DU|^2\rho'\,dxdy + c(\varepsilon)\iint U^2\rho'dxdy +\iint \Psi_x^2\rho\,dxdy. 
\end{multline*}
Note that similarly to \eqref{4.2}
\begin{equation}\label{4.11}
\|\Psi_x\|_{L_2(0,T;L_{2,+}^{\rho(x)})} \leq c(T)\|\mu-\widetilde\mu\|_{H^{1/3,1}(B_T)}
\end{equation}
since $\Psi=0$ for $x\geq 2$. Then the desired result succeeds from inequality \eqref{4.7}.
\end{proof}

\begin{theorem}\label{T4.3}
Let $\rho(x)$ be an admissible weight function, such that $\rho'(x)$ is also an admissible weight function and $\rho^{1/3}(x)\leq c_0\rho'(x) \ \forall x\geq 0$ for certain positive constant $c_0$. Then for any $T>0$ and $M>0$ there exist constant $c=c(T,M)$, such that for any two weak solutions $u(t,x,y)$ and $\widetilde u(t,x,y)$ to problem \eqref{1.1}--\eqref{1.4}, satisfying $\|u\|_{X_w^{1,\rho(x)}(\Pi_T^+)}, \|\widetilde u\|_{X_w^{1,\rho(x)}(\Pi_T^+)} \leq M$, with corresponding data $u_0, \widetilde u_0\in \widetilde H_+^{1,\rho(x)}$, $\mu, \widetilde\mu\in \widetilde H^{2/3,2}(B_T)$, $f, \widetilde f\in L_2(0,T;\widetilde H_+^{1,\rho(x)})$, $u_0(0,y)\equiv \mu(0,y)$, $\widetilde u_0(0,y)\equiv \widetilde \mu(0,y)$, the following inequality holds:
\begin{equation}\label{4.12}
\|u -\widetilde u\|_{X_w^{1,\rho(x)}(\Pi_T^+)} \leq c\bigl( \|u_0 - \widetilde u_0\|_{H_+^{1,\rho(x)}} +
\|\mu-\widetilde\mu\|_{H^{2/3,2}(B_T)} +
\|f-\widetilde f\|_{L_2(0,T;H_+^{1,\rho(x)})}\bigr).
\end{equation} 
\end{theorem}

\begin{proof}
First of all note that the hypothesis of Theorem~\ref{T4.2} holds and, consequently, inequality \eqref{4.1} is satisfied.

Introduce the same functions $\Psi$, $U_0$, $F$, $U$ as in the proof of Theorem~\ref{T4.1}. Note that
\begin{gather*}
\|\Psi\|_{X^{1,\rho(x)}(\Pi_T^+)} \leq c(T)\|\mu-\widetilde\mu\|_{H^{2/3,2}(B_T)}, \\
\|U_0\|_{H_+^{1,\rho(x)}} \leq \|u_0-\widetilde u_0\|_{H_+^{1,\rho(x)}} +
c(T)\|\mu-\widetilde\mu\|_{H^{2/3,2}(B_T)}, \\
\|F\|_{L_2(0,T;H_+^{1,\rho(x)})} \leq \|f-\widetilde f\|_{L_2(0,T;H_+^{1,\rho(x)})} +
c(T)\|\mu-\widetilde\mu\|_{H^{2/3,2}(B_T)}.
\end{gather*}
Apply Lemma~\ref{L2.13}. Note that since $\rho^2(\rho')^{-1} \leq c \rho^{3/2}(\rho')^{1/2}$
\begin{multline*}
\iint u^2u_x^2 \rho^2(\rho')^{-1}\,dxdy \leq c \Bigl(\iint u^4\rho^2\,dxdy \iint u_x^4\rho'\rho\,dxdx\Bigr)^{1/2} \\ \leq c_1\iint (|Du_x|^2\rho' +u_x^2\rho)\,dxdy.
\end{multline*}
In particular, $uu_x\in L_2(0,T;L_{2,+}^{\rho^2(x)/\rho'(x)})$. Then inequality \eqref{2.52} for $\gamma\equiv 1$ (together with \eqref{1.17}) yields that for $t\in (0,T]$
\begin{multline}\label{4.13}
\iint (U_x^2+U_y^2)\rho\,dxdy + \int_0^t\!\! \iint |D^2 U|^2\rho'\,dxdyd\tau \leq 
\iint (U_{0x}^2+U_{0y}^2)\rho\,dxdy \\ +c\int_0^t \!\!\iint (U_x^2+U_y^2)\rho\,dxdyd\tau +
\int_0^t \!\!\iint (F_x^2+F_y^2+F^2)\rho\,dxdyd\tau \\ +
2\int_0^t\!\! \iint (uu_x-\widetilde u \widetilde u_x)[(U_x\rho)_x+U_{yy}\rho]\,dxdyd\tau.
\end{multline}
The last integral in the right side of \eqref{4.13} is not greater than
\begin{multline*}
\varepsilon \int_0^t\!\!\iint (U_{xx}^2+U_{yy}^2+U_x^2)\rho'\,dxdyd\tau \\ +
c(\varepsilon) \int_0^t\!\!\iint \bigl[(u_x^2+\widetilde u_x^2)(U^2+\Psi^2) +
(u^2+\widetilde u^2)(U_x^2+\Psi_x^2)\bigr]\frac{\rho^2}{\rho'}\,dxdyd\tau,
\end{multline*}
where $\varepsilon>0$ can be chosen arbitrarily small. 
Here again since $ \rho^2(\rho')^{-1} \leq c \rho^{3/2}(\rho')^{1/2}$
\begin{multline*}
\iint u_x^2 (U^2+\Psi^2)\frac{\rho^2}{\rho'}\,dxdy  \leq
c\Bigl(\iint u_x^4 \rho'\rho\,dxdy \iint (U^4+\Psi^4)\rho^2\,dxdy\Bigr)^{1/2} \\ \leq
c_1\Bigl[\Bigr(\iint |Du_x|^2\rho'\,dxdy \iint u_x^2\rho\,dxdy\Bigr)^{1/2} +
\iint u_x^2\rho\,dxdy\Bigr] \\ \times
\Bigl[\Bigr(\iint \bigl(|DU|^2+|D\Psi|^2\bigr)\rho\,dxdy \iint (U^2+\Psi^2)\rho\,dxdy\Bigr)^{1/2} +
\iint (U^2+\Psi^2)\rho\,dxdy\Bigr] \\ \leq c_2
\iint \bigl(|DU|^2+|D\Psi|^2\bigr)\rho\,dxdy + c_2
\Bigl[\iint |Du_x|^2\rho'\,dxdy+1\Bigr]\iint (U^2+\Psi^2)\rho\,dxdy,
\end{multline*}
where the first multiplier in the last term belongs to the space $L_1(0,T)$ and the second one is estimated uniformly with respect to $t$ according to \eqref{4.1} and \eqref{4.2}. Finally,
\begin{multline*}
\iint u^2(U_x^2+\Psi_x^2)\frac{\rho^2}{\rho'}\,dxdy   \leq
c\Bigl(\iint u^4\rho^2\,dxdy \iint (U_x^4+\Psi_x^4)\rho'\rho\,dxdy\Bigr)^{1/2}  \\ \leq 
c_1\Bigr(\iint \bigl(|DU_x|^2+|D\Psi_x|^2\bigr)\rho'\,dxdy \iint (U_x^2+\Psi_x^2)\rho\,dxdy\Bigr)^{1/2} +
\iint (U_x^2+\Psi_x^2)\rho\,dxdy.
\end{multline*}
As a result, the statement of the theorem follows from inequality \eqref{4.13}.
\end{proof}

\begin{remark}\label{R4.2}
Theorems~\ref{T1.2}, \ref{T4.2} and \ref{T4.3} show that under the hypothesis of Theorem~\ref{T1.2} and additional assumption  $\rho^{1/3}(x)\leq c_0\rho'(x) \ \forall x\geq 0$  problem \eqref{1.1}--\eqref{1.4} is globally well-posed in the space $X_w^{1,\rho(x)}(\Pi_T^+)$. This additional assumption holds for any exponential weight $e^{2\alpha x}$, $\alpha>0$, and for the power weight $(1+x)^{2\alpha}$ if $\alpha\geq 3/4$.
\end{remark}

For regular solutions we prefer to present well-posedness in another form.

\begin{theorem}\label{T4.4}
Let $T>0$ and $\rho(x)$ be an admissible weight function, such that $\rho'(x)$ is also an admissible weight function and $\rho'(x)\geq 1 \ \forall x\geq 0$. Denote by $\mathcal F$ the space of functions $f(t,x,y)$, defined on $\Pi_T^+$ and satisfying the hypothesis of Theorem~\ref{T1.3}, endowed with the natural norm. Then the mapping $(u_0,\mu,f)\mapsto u$, where $u$ is the corresponding solution of problem \eqref{1.1}--\eqref{1.4} and $u_0(0,y)\equiv \mu(0,y)$, is Lipschitz continuous on any ball in the norm of the mapping $\widetilde H_+^{3,\rho(x)} \times \widetilde H^{4/3,4}(B_T) \times \mathcal F \to X^{3,\rho(x)}(\Pi_T^+)$.
\end{theorem}

\begin{proof}
Let $M>0$, let the functions $u_0$,$\mu$,$f$ satisfy the hypothesis of Theorem~\eqref{T1.3} and $\|(u_0,\mu,f)\|_{\widetilde H_+^{3,\rho(x)} \times \widetilde H^{4/3,4}(B_T) \times \mathcal F}\leq M$, then it follows from \eqref{3.51} that $\|u\|_{X^{3,\rho(x)}(\Pi_T^+)}\leq c_0(M)$. Define the functions $\psi$ and $U$ by formulas \eqref{3.3} and \eqref{3.5}. Let the triplet $(\widetilde u_0, \widetilde\mu,\widetilde f)$ be another one satisfying the same assumptions, define similarly the functions $\widetilde \psi$ and $\widetilde U$. Then similarly to \eqref{3.39}--\eqref{3.50} for $t_0\in (0,T]$
\begin{multline*}
\|U-\widetilde U\|_{X^{3,\rho(x)}(\Pi_{t_0}^+)} \leq c(M)\bigl(\|u_0-\widetilde u_0\|_{\widetilde H_+^{3,\rho(x)}} +
\|\mu-\widetilde\mu\|_{\widetilde H^{4/3,4}(B_T)} +\|f-\widetilde f\|_{\mathcal F}  \\+
t_0^{1/2}\|U-\widetilde U\|_{X^{3,\rho(x)}(\Pi_{t_0}^+)}\bigr).
\end{multline*}
Taking into account also that $\|\psi-\widetilde\psi\|_{X^{3,\rho(x)}(\Pi_T^+)} \leq c(T)\|\mu-\widetilde\mu\|_{\widetilde H^{4/3,4}(B_T)}$ we finish the proof by the standard argument.
\end{proof}

\section{Large-time decay of small solutions}\label{S5}

\begin{proof}[Proof of Theorem~\ref{T1.4}]
Let $\alpha>0$, $\rho(x)\equiv e^{2\alpha x}$, $u_0\in \widetilde H_+^{1,\rho(x)}$, $u_0(0,y)\equiv 0$, $\mu\equiv 0$, $f\equiv 0$. Consider the solution to problem \eqref{1.1}--\eqref{1.4} (in the cases a) and c)) $u\in X_w^{1,\rho(x)}(\Pi_T^+)\ \forall T$. Note that $uu_x\in L_2(0,T;L_{2,+}^{\rho(x)})$ (see, for example, \eqref{3.22}).

Apply Lemma~\ref{L2.12}, then equality \eqref{2.47} for $\rho\equiv 1$ provides, in fact, the conservation law \eqref{1.5}, in particular,
\begin{equation}\label{5.1}
\|u(t,\cdot,\cdot)\|_{L_{2,+}}\leq \|u_0\|_{L_{2,+}} \quad \forall t\geq 0.
\end{equation} 
Next, write down equality \eqref{2.47} for $\rho\equiv e^{2\alpha x}$:
\begin{multline}\label{5.2}
\iint u^2\rho\,dxdy + \iint_{B_t} u_x^2\big|_{x=0}\,dyd\tau +
2\alpha \int_0^t \!\!\iint (3u_x^2+u_y^2)\rho\,dxdyd\tau \\
-2\alpha(b+4\alpha^2) \int_0^t \!\! \iint u^2\rho\,dxdyd\tau
= \iint u_0^2\rho\,dxdy +
\frac{2\alpha}3 \int_0^t\!\!\iint u^3\rho\,dxdyd\tau.
\end{multline} 
Since $u^3\rho\in L_\infty(0,T;L_{1,+})$ equality \eqref{5.2} provides the following equality in a differential form:
for a.e. $t>0$
\begin{multline}\label{5.3}
\frac{d}{dt}\iint u^2\rho\,dxdy + \int_0^L u_x^2\big|_{x=0}\,dy +
2\alpha \iint (3u_x^2+u_y^2)\rho\,dxdy \\
-2\alpha(b+4\alpha^2) \iint u^2\rho\,dxdy
= \frac{2\alpha}3 \iint u^3\rho\,dxdy.
\end{multline} 
Continuing inequality \eqref{3.17}, we find with the use of \eqref{5.1} that uniformly with respect to $L$
\begin{equation}\label{5.4}
\frac{2}3 \iint u^3\rho\,dxdy \leq \frac12 \iint |Du|^2\rho\,dxdy  +
c (\|u_0\|_{L_{2,+}}+\|u_0\|_{L_{2,+}}^2)\iint u^2\rho\,dxdy.
\end{equation}
Inequalities \eqref{1.19} or \eqref{1.20} yield that for certain constant $c_0$
\begin{equation}\label{5.5}
\frac12 \iint u_y^2\rho\,dxdy \geq \frac{c_0}{L^2} \iint u^2\rho\,dxdy.
\end{equation}
Combining \eqref{5.3}--\eqref{5.5} we find that uniformly with respect to $\alpha$ and $L$
\begin{multline}\label{5.6}
\frac{d}{dt}\iint u^2\rho\,dxdy + \int_0^L u_x^2\big|_{x=0}\,dy +
\alpha \iint |Du|^2\rho\,dxdy \\
+\alpha\Bigl(\frac{c_0}{L^2}-2b-8\alpha^2-c (\|u_0\|_{L_{2,+}}+\|u_0\|_{L_{2,+}}^2)\Bigr) \iint u^2\rho\,dxdy
\leq 0.
\end{multline}
Choose $\displaystyle L_0= \frac12\sqrt{\frac{c_0}b}$ if $ b>0$, $\displaystyle \alpha_0 = \frac{\sqrt{c_0}}{8L}$, $\epsilon>0$ satisfying an inequality $\displaystyle \epsilon_0+\epsilon_0^2 \leq \frac {c_0}{8cL^2}$, $\beta = \displaystyle \frac{c_0}{4L^2}$. Then it follows from \eqref{5.6} that
\begin{equation}\label{5.7}
\frac{d}{dt}\iint u^2\rho\,dxdy + \int_0^L u_x^2\big|_{x=0}\,dy +
\alpha \iint |Du|^2\rho\,dxdy 
+\alpha\beta\iint u^2\rho\,dxdy
\leq 0.
\end{equation}
In particular, inequality \eqref{5.7} provides estimate \eqref{1.13} if  $u_0\in \widetilde H_+^{1,\rho(x)}$, $u_0(0,y)\equiv 0$. In the general case $u_0\in L_{2,+}^{\rho(x)}$ this estimate is obtained via closure with the use of Theorem~\ref{T4.1}. 

Moreover, since inequality \eqref{5.7} can be written in a form
$$
\frac{d}{dt}\Bigl[e^{\alpha\beta t} \iint u^2\rho\,dxdy\Bigr] +
e^{\alpha\beta t}\Bigl[\int_0^L u_x^2\big|_{x=0}\,dy +
\alpha \iint |Du|^2\rho\,dxdy\Bigr] \leq 0,
$$ 
we find (again if $u_0\in \widetilde H_+^{1,\rho(x)}$, $u_0(0,y)\equiv 0$) that
\begin{equation}\label{5.8}
\int_0^t e^{\alpha\beta \tau} \Bigl[\int_0^L u_x^2\big|_{x=0}\,dy +
\alpha \iint |Du|^2\rho\,dxdy\Bigr]\,d\tau \leq \|u_0\|^2_{L_{2,+}^{\rho(x)}}.
\end{equation}

Write down inequality \eqref{2.52} for $\gamma(t)\equiv e^{\alpha\beta t}$  and $f_1\equiv -uu_x$, then taking into account \eqref{5.8} we derive the following inequality:
\begin{multline}\label{5.9}
e^{\alpha\beta t}\iint (u_x^2+u_y^2)\rho\,dxdy +
2\alpha \int_0^t e^{\alpha\beta\tau} \iint (3u_{xx}^2+4u_{xy}^2+u_{yy}^2)\rho \,dxdyd\tau  \\ 
\leq c+ 2\int_0^t e^{\alpha\beta\tau} \iint uu_x\bigl[(u_x\rho)_x +u_{yy}\rho\bigr] \,dxdyd\tau.
\end{multline} 
Differentiate the corresponding equality \eqref{2.54} (for $f\equiv -uu_x$), multiply by $e^{\alpha\beta t}$ and integrate with respect to $t$:
\begin{multline}\label{5.10}
-\frac13 e^{\alpha\beta t} \iint u^3\rho\,dxdy +
2\int_0^t e^{\alpha\beta\tau} \iint uu_x(u_{xx} +u_{yy})\rho \,dxdyd\tau \\
+2\alpha\int_0^t e^{\alpha\beta\tau} \iint u^2(u_{xx} +u_{yy})\rho \,dxdyd\tau =
-\frac13 \iint u_0^3\rho\,dxdy  \\ 
-\frac{\alpha}3(\beta +2b) \int_0^t e^{\alpha\beta\tau}
\iint u^3\rho\,dxdyd\tau -
\frac{\alpha}2 \int_0^t e^{\alpha\beta\tau}\iint u^4\rho\,dxdyd\tau.
\end{multline}
Summing \eqref{5.9} and \eqref{5.10} we find that
\begin{multline}\label{5.11}
e^{\alpha\beta t}\iint (u_x^2+u_y^2-\frac13 u^3)\rho\,dxdy +
2\alpha \int_0^t e^{\alpha\beta\tau} \iint (3u_{xx}^2+4u_{xy}^2+u_{yy}^2)\rho \,dxdyd\tau  \\ 
\leq c+ 2\alpha\int_0^t e^{\alpha\beta\tau} \iint uu_x^2\rho\,dxdyd\tau - 
2\alpha\int_0^t e^{\alpha\beta\tau} \iint u^2(u_{xx} +u_{yy})\rho \,dxdyd\tau \\
-\frac{\alpha}3(\beta +2b) \int_0^t e^{\alpha\beta\tau}
\iint u^3\rho\,dxdyd\tau.
\end{multline} 
Estimating the integrals in the right side of \eqref{5.11} with the help of \eqref{3.29}--\eqref{3.32}, \eqref{1.13} and \eqref{5.8} yields:
$$
e^{\alpha\beta t}\iint (u_x^2+u_y^2-\frac13 u^3)\rho\,dxdy \leq c,
$$
where 
$$
\frac13 \iint u^3\rho\,dxdy \leq c\Bigl[\Bigl(\iint |Du|^2\rho\,dxdy\Bigr)^{1/2}
\iint u^2\rho\,dxdy +\Bigl(\iint u^2\rho\,dxdy\Bigr)^{3/2}\Bigr],
$$
and \eqref{1.14} follows.
\end{proof}

\end{document}